\documentclass[10pt]{amsart}
\usepackage{amssymb}
\usepackage{amsmath}
\usepackage{amsfonts}
\usepackage{hyperref}
\usepackage[arrow,curve,matrix]{xy}

\newcommand{\reg}{{\rm reg}}
\newcommand{\Pic}{{\rm Pic}}

\newcommand{\gon}{{\rm gon}}
\newcommand{\Cliff}{{\rm Cliff}}

\renewcommand{\det}{{\rm det}}

\renewcommand{\dim}{{\rm dim}}
\renewcommand{\deg}{{\rm deg}}

\newcommand{\rk}{{\rm rk}}
\newcommand{\coker}{{\rm coker}}
\renewcommand{\det}{{\rm det}}

\theoremstyle{plain}

\newtheorem{thm}{Theorem}[section]

\newtheorem{cor}[thm]{Corollary}
\newtheorem{prop}[thm]{Proposition}

\newtheorem{con}[thm]{Conjecture}

\theoremstyle{definition}
\newtheorem{defn}[thm]{Definition}
\newtheorem{ex}[thm]{Example}
\newtheorem{rmk}[thm]{Remark}
\newtheorem{question}[thm]{Question}
\newtheorem{problem}[thm]{Problem}

\def\PP{{\textbf P}}
\def\OO{\mathcal{O}}

\def\cB{\mathcal{B}}
\def\cA{\mathcal{A}}
\def\F{\mathcal{F}}

\def\E{\mathcal{E}}

\def\K{\mathcal{K}}

\def\I{\mathcal{I}}

\def\cM{\mathcal{M}}
\def\cR{\mathcal{R}}
\def\rr{\overline{\mathcal{R}}}
\def\cZ{\mathcal{Z}}

\def\mm{\overline{\mathcal{M}}}
\def\kk{\overline{\mathcal{K}}}

\def\rem{\overline{\textbf{M}}}

\def\pem{\widetilde{\textbf{M}}}

\begin{document}
\title{Koszul cohomology and applications to moduli}
\author[M. Aprodu]{Marian Aprodu}
\address{Institute of Mathematics "Simion Stoilow" of the Romanian
Academy, RO-014700 Bucharest,
and \c Scoala Normal\u a Superioar\u a Bucure\c sti,
Calea Grivi\c tei 21,  RO-010702 Bucharest, Romania}
 \email{{\tt
aprodu@imar.ro}}
\author[G. Farkas]{Gavril Farkas}

\address{Humboldt-Universit\"at zu Berlin, Institut F\"ur Mathematik,
10099 Berlin, Germany}
\email{{\tt farkas@math.hu-berlin.de}}

\markboth{M. APRODU and G. FARKAS}
{Koszul cohomology with applications to moduli}

\thanks{Both authors would like to thank the Max Planck Institut f\"ur
Mathematik Bonn for hospitality during the preparation of this work.
Research of the first author partially supported by a PN-II-ID-PCE-2008-2
grant, cod CNCSIS 2228.
Research  of the second author partially supported by an Alfred P. Sloan Fellowship}

\maketitle

%

\section{Introduction}

One of the driving problems in the theory of algebraic curves in the past two decades has been
\emph{Green's Conjecture} on syzygies of canonical curves. Initially formulated by M. Green  \cite{Gr84a}, it is a deceptively simple vanishing statement concerning Koszul cohomology
groups of canonical bundles of curves: If $C\stackrel{|K_C|}\longrightarrow \mathbb P^{g-1}$ is a smooth canonically embedded curve of genus $g$, Green's Conjecture predicts the equivalence
\begin{equation}\label{grr}
K_{p, 2}(C, K_C)=0\Longleftrightarrow p<\mathrm{Cliff}(C),
\end{equation}
where $\mathrm{Cliff}(C):=\mbox{min}\{\mbox{deg}(L)-2r(L): L\in \mathrm{Pic}(C), \ h^i(C, L)\geq 2, \ i=0, 1\}$
denotes the \emph{Clifford index} of $C$. The main attraction of Green's Conjecture is that it links the \emph{extrinsic} geometry of $C$ encapsulated in $\mathrm{Cliff}(C)$ and all the linear series $\mathfrak g^r_d$ on $C$, to the \emph{intrinsic} geometry (equations) of the canonical embedding. In particular, quite remarkably, it shows that one can read the Clifford index of any curve off the equations of its canonical embedding. Hence  in some sense, a curve has no other interesting line bundle apart from the canonical bundle
and its powers\footnote{{ "The canonical bundle is not called canonical for nothing"- Joe Harris}}.

One implication in (\ref{grr}), namely that $K_{p, 2}(C, K_C)\neq 0$ for $p\geq \mbox{Cliff}(C)$, having been immediately established in \cite{GL84}, see also Theorem \ref{thm: GL nonvan} in this survey,  the converse, that is, the vanishing statement  $$K_{p, 2}(C, K_C)=0
\ \mbox{ for } p<\mbox{Cliff}(C),$$ attracted a great deal of effort and resisted proof despite an
amazing number of attempts and techniques devised to prove it, see \cite{GL84}, \cite{Sc1}, \cite{Sc3}, \cite{Ein}, \cite{Ei92}, \cite{Paranjape-Ramanan}, \cite{Teixidor02}, \cite{Voisin: Proc LMS}. The major breakthrough came around 2002 when Voisin \cite{Voisin: even}, \cite{Voisin: odd}, using specialization to curves on $K3$ surfaces, proved that Green's Conjecture holds for a general curve $[C]\in \cM_g$:
\begin{thm}\label{vo}
For a general curve $[C]\in \cM_{2p+3}$ we have that $K_{p, 2}(C, K_C)=0$.
For a general curve $[C]\in \cM_{2p+4}$ we have that $K_{p, 2}(C, K_C)=0$. It follows that Green's
Conjecture holds for general curves of any genus.
\end{thm}
Combining the results of Voisin with those of Hirschowitz and Ramanan \cite{HR98}, one finds that Green's Conjecture is true for \emph{every} smooth curve $[C]\in \cM_{2p+3}$ of maximal gonality $\mbox{gon}(C)=p+3$. This turns out to be
a remarkably strong result. For instance, via a specialization argument,  Green's Conjecture for curves of \emph{arbitrary} curves of maximal gonality  implies Green's Conjecture for \emph{general} curves of
genus $g$ and \emph{arbitrary} gonality $2\leq d\leq [g/2]+2$. One has the following more precise result cf. \cite{Ap05}, see Theorem \ref{thm: small BN} for a slightly different proof:

\begin{thm}\label{aprodu1}
We fix integers $2\leq d\leq [g/2]+1$. For any smooth $d$-gonal curve $[C]\in \cM_g$ satisfying the condition
$$\mathrm{dim }\  W^1_{g-d+2}(C)\leq g-d+2,$$
we have that $K_{d-3, 2}(C, K_C)=0$. In particular $C$ satisfies Green's Conjecture.
\end{thm}

Dimension theorems from  Brill-Noether theory
due to Martens, Mumford, and Keem cf. \cite{ACGH},
indicate precisely when the condition appearing
in the statement of Theorem \ref{aprodu1} is
verified. In particular, Theorem \ref{aprodu1} proves Green's Conjecture for general
$d$-gonal curves of genus $g$ for any possible gonality $2\leq d\leq [g/2]+2$ and
offers an alternate, unitary proof of classical results of Noether,
Enriques-Babbage-Petri as well as of more recent results due to
Schreyer and Voisin. It also implies the following new result
which can be viewed as a proof of statement (\ref{grr}) for $6$-gonal curves.
We refer to Subsection 4.1 for details:

\begin{thm}
For any curve $C$ with $\mathrm{Cliff}(C)\ge 4$, we
have $K_{3,2}(C,K_C)= 0$. In particular, Green's Conjecture holds for arbitrary $6$-gonal curves.
\end{thm}

Theorem \ref{vo} can also be applied to solve various related
problems. For instance, using precisely Theorem \ref{vo},
the {\em  Green-Lazarsfeld Gonality Conjecture} \cite{GL86} was verified
for general $d$-gonal curves, for any $2\leq d\leq (g+2)/2$, cf. \cite{AV}, \cite{Ap05}. In a few words, this
conjecture states that the gonality of any curve can be read
off the Koszul cohomology with values in line bundles of
large degree, such as the powers of the canonical bundle.
We shall review all these results in
Subsection 4.2.

\medskip

Apart from surveying the progress made on Green's and the Gonality
Conjectures, we discuss a number of new conjectures for
syzygies of line bundles on curves. Some of these conjectures have already appeared in print (e.g. the \emph{Prym-Green
Conjecture} \cite{FaL08}, or the syzygy conjecture for special line bundles on general curves \cite{Fa06a}), others like the \emph{Minimal Syzygy Conjecture} are new and have never been formulated before.

For instance we propose the following refinement of the Green-Lazarsfeld Gonality Conjecture \cite{GL86}:
\begin{con}\label{eta}
Let $C$ be a general curve of genus $g=2d-2\geq 6$
and $\eta\in\Pic^0(C)$ a general line bundle.
Then $K_{d-4,2}(C,K_C\otimes \eta)= 0$.
\end{con}
Conjecture \ref{eta} is the sharpest vanishing statement one can make for general line bundles of degree $2g-2$ on a curve of genus $g$. Since
$$\mbox{dim } K_{d-4, 2}(C, K_C\otimes \eta)=\mbox{dim } K_{d-3, 1}(K_C\otimes \eta),$$
it follows that the failure locus of Conjecture \ref{eta} is a virtual divisor in the
universal degree $0$ Picard variety $\mathfrak{Pic}_g^0\rightarrow \cM_g$. Thus it predicts that
the non-vanishing locus $$\{[C, \eta]\in \mathfrak{Pic}_g^0: K_{d-4, 2}(C, K_C\otimes \eta)\neq 0\}$$
is an "honest" divisor on $\mathfrak{Pic}_g^0$.  Conjecture \ref{eta} is also sharp in the sense that from
the Green-Lazarsfeld Non-Vanishing Theorem \ref{thm: nonvan} it follows that  $K_{d-4, 1}(C, K_C\otimes \eta)\neq 0$. Similarly, always $K_{d-3,2}(C, K_C\otimes \eta)\neq 0$ for all $[C, \eta]\in \mathfrak{Pic}_g^0$. A yet stronger conjecture is the following vanishing statement for $l$-roots of trivial bundles on curves:

\begin{con}\label{levelp}
Let $C$ be a general curve of genus $g=2d-2\geq 6$. Then for every prime $l$ and every line
bundle $\eta\in \mathrm{Pic}^0(C)-\{\OO_C\}$ satisfying $\eta^{\otimes l}=\OO_C$, we have that
$K_{d-4, 2}(C, K_C\otimes \eta)=0$.
\end{con}

In order to prove Conjecture \ref{levelp} it suffices to exhibit a single pair $[C, \eta]$ as above, for which
$K_C\otimes \eta\in \mbox{Pic}^{2g-2}(C)$ satisfies property $(N_{d-4})$. The case most studied so far is that of level $l=2$, when one recovers the \emph{Prym-Green Conjecture} \cite{FaL08}
which has been checked using Macaulay2 for $g\leq 14$. The Prym-Green Conjecture is a subtle statement which for small values of $g$ is equivalent to the Prym-Torelli Theorem, again see \cite{FaL08}.  Since the non-vanishing condition $K_{d-4, 2}(C, K_C\otimes \eta)\neq 0$ is divisorial in moduli,  Conjecture \ref{levelp} is of great help in the study of the birational geometry of the  compactification $\rr_{g, l}:=\mm_g(\mathcal{B}\mathbb Z_l)$ of the moduli space  $\mathcal{R}_{g, l}$ classifying pairs $[C, \eta]$, where $[C]\in \cM_g$ and $\eta\in \mbox{Pic}^0(C)$ satisfies $\eta^{\otimes l}=\OO_C$.
\vskip 5pt

Concerning both Conjectures \ref{eta} and \ref{levelp}, it is an open
problem to find an analogue of the Clifford index of the curve, in the sense that the classical Green Conjecture is not only a Koszul cohomology vanishing statement but also allows one to read off the Clifford index from a non-vanishing statement for $K_{p, 2}(C, K_C)$. It is an interesting open question to find a \emph{Prym-Clifford index}
playing the same role like the original $\mbox{Cliff}(C)$ in (\ref{grr}) and to describe it in terms of the corresponding Prym varieties: Is there a geometric characterization of those Prym varieties $\mathcal{P}_g([C, \eta])\in \cA_{g-1}$ corresponding to pairs $[C, \eta]\in \cR_g$ with $K_{p, 2}(C, K_C\otimes \eta)\neq 0$?
\vskip 4pt

Another recent development on syzygies of curves came from a completely different direction, with the realization that loci in the moduli space $\cM_g$
consisting of curves having exceptional syzygies, can be used effectively to answer questions about the birational
geometry of the moduli space $\mm_g$ of stable curves of genus $g$, cf. \cite{FaPo05}, \cite{Fa06a}, \cite{Fa06b},
in particular, to produce infinite series of effective divisors on $\mm_g$ violating the Harris-Morrison Slope Conjecture \cite{HaMo}. We recall that the slope $s(D)$ of an  effective divisor $D$ on
$\mm_g$ is defined as the smallest rational number $a/b$ with $a, b\geq 0$,
such that the class
$a \lambda-b(\delta_0+\cdots+\delta_{[g/2]})-[D]\in \mathrm{Pic}(\mm_g)$ is an effective
$\mathbb Q$-combination of boundary divisors. The Slope Conjecture \cite{HaMo} predicts a lower bound for the slope of effective divisors on $\mm_g$
 $$s(\mm_g):=\mbox{inf}_{D\in \mbox{Eff}(\mm_g)} \ s(D)\geq 6+\frac{12}{g+1}$$
with equality precisely when $g+1$ is composite; the quantity $6+12/(g+1)$ is the slope of the Brill-Noether divisors on $\mm_g$, in case such divisors exist.
A first counterexample to the Slope Conjecture was found in \cite{FaPo05}: The
locus $$\K_{10}:=\{[C]\in \cM_g: C \mbox{ lies on a } K3 \mbox{ surface}\}$$
can be interpreted as being set-theoretically equal to the locus of curves $[C]\in \cM_{10}$ carrying a linear series
$L \in W^4_{12}(C)$ such that the multiplication map
$$\nu_2(L):\mbox{Sym}^2 H^0(C, L)\rightarrow H^0(C, L^{\otimes 2})$$ is not an isomorphism, or equivalently
$K_{0, 2}(C, L)\neq 0$. The main advantage of this  Koszul-theoretic description is that it provides a characterization of the $K3$ divisor $\K_{10}$ in a way that makes no reference to $K3$ surfaces and can be easily generalized to other genera. Using this characterization  one shows that $s(\kk_{10})=7<78/11$, that is,
$\kk_{10}\in \mbox{Eff}(\mm_{10})$ is a counterexample to the Slope Conjecture.

 Koszul cohomology provides an effective way
of constructing cycles on $\cM_g$. Under suitable numerical
conditions, loci of the type
$$\cZ_{g, 2}:=\{[C]\in \cM_g: \exists L\in W^r_d(C) \mbox{ such that } K_{p, 2}(C, L)\neq 0\}$$ are virtual divisors on $\cM_g$,
that is, degeneracy loci of morphisms between vector bundles of the same rank over $\cM_g$. The problem of extending
these vector bundes over $\mm_g$ and computing the virtual classes of the resulting degeneracy loci is in general daunting, but has been solved successfully in the case $\rho(g, r, d)=0$, cf. \cite{Fa06b}.
Suitable vanishing statements of the Koszul cohomology for general curves (e.g. Conjectures \ref{eta}, \ref{strongmaxrank}) show that, when applicable, these virtual Koszul divisors are actual divisors and they are
quite useful in specific problems such as the Slope Conjecture or showing that certain moduli spaces
of curves (with or without level structure) are of general type, see \cite{Fa08}, \cite{FaL08}.
A picturesque application of the Koszul technique in the study of parameter spaces is the following result about the birational type
of the moduli space of Pryms varieties $\overline{\mathcal{R}}_g=\overline{\mathcal{R}}_{g, 2}$, see \cite{FaL08}:
\begin{thm}\label{fl}
The moduli space $\overline{\mathcal{R}}_g$ is of general type for $g\geq 13$ and $g\neq 15$.
\end{thm}

The proof of Theorem \ref{fl} depends on the parity of $g$. For $g=2d-2$, it boils down to calculating the class of the compactification in $\rr_g$ of the failure
locus of the Prym-Green Conjecture, that is, of the locus $\{[C, \eta]\in \cR_{2d-2}: K_{d-4, 2}(C, K_C\otimes \eta)\neq 0\}$. For odd $g=2d-1$, one computes the class of a "mixed" Koszul cohomology locus in $\rr_g$ defined in terms of Koszul cohomology groups of $K_C$ with values in $K_C\otimes \eta$.
\vskip 8pt

The outline of the paper is as follows. In Section  2 we review
the definition of Koszul cohomology as introduced by M. Green \cite{Gr84a}
and discuss basic facts. In Section 3 we recall the construction of (virtual) Koszul cycles on $\mm_g$ following \cite{Fa06a} and \cite{Fa06b} and explain how their cohomology classes can be calculated. In Section 4 we discuss a number of conjectures on syzygies of curves, starting with Green's Conjecture and the Gonality Conjecture and continuing with the Prym-Green Conjecture. We end by proposing in Section 5 a strong version of the Maximal Rank Conjecture.

Some results stated in \cite{Ap04}, \cite{Ap05}
are discussed here in greater detail.
Other results are new (see Theorems \ref{thm: Green large genus} and \ref{thm: GL}).

\section{Koszul cohomology}

\subsection{Syzygies}

Let $V$ be an $n$-dimensional complex vector space, $S:=S(V)$ the symmetric
algebra of $V$, and $X\subset\mathbb{P}V^\vee:=\mbox{\rm Proj}(S)$
a non-degenerate subvariety, and denote by $S(X)$ the
homogeneous coordinate ring of $X$. To the embedding of
$X$ in $\mathbb{P} V^\vee$, one associates the
{\em Hilbert function}, defined by
$$
h_X(d):=\mbox{dim}_{\mathbb C}\left(S(X)_d\right)
$$
for any positive integer $d$. A remarkable property of $h_X$
is its polynomial behavior for large values of $d$. It is a consequence
of the existence of a {\em graded
minimal resolution} of the $S$-module $S(X)$, which is an exact sequence
$$
0\rightarrow E_s\rightarrow\ldots\rightarrow E_2\rightarrow E_1\rightarrow S\rightarrow S(X)\rightarrow 0
$$
with
$$
E_p=\mathop\bigoplus\limits_{j>p}S(-j)^{\beta_{pj}(X)}.
$$

The Hilbert function of $X$ is then given by
\begin{equation}
\label{eqn: Hilbert}
h_X(d)=\mathop\sum\limits_{p,j} (-1)^p\beta_{pj}(X) \left(\begin{array}{c}{n+d-j}\\n\end{array}\right),
\end{equation}
where
$$
\left(\begin{array}{c}{t}\\n\end{array}\right)=\frac{t(t-1)\dots(t-n+1)}{n\; !},\; \mbox{\rm for any } t\in
{\bf R},
$$
and note that the expression on the right-hand-side is polynomial for large
$d$. This reasoning yields naturally to the definition of {\em syzygies of $X$},
which are the graded components of the graded $S$-modules
$E_p$. The integers $$\beta_{pj}(X)=\mbox{dim}_{\mathbb C} \mbox{Tor}^j(S(X), \mathbb C)_p$$
are called the {\em graded Betti
numbers of $X$} and determine completely the Hilbert function,
according to formula~(\ref{eqn: Hilbert}). Sometimes we also write $b_{i, j}(X):=\beta_{i, i+j}(X)$.

One main difficulty in developing  syzygy theory was to
find effective geometric methods for computing these invariants.
In the eighties, M. Green and R. Lazarsfeld published a series of papers,
\cite{Gr84a}, \cite{Gr84b}, \cite{GL84},
\cite{GL86} that shed a new light on syzygies.
Contrary to the classical point of view, they
look at integral closures of the homogeneous coordinates rings,
rather than at the rings themselves. This
approach, using intensively the language of Koszul cohomology,
led to a number of beautiful geometrical results with numerous applications in classical algebraic geometry
as well as moduli theory.

\subsection{Definition of Koszul cohomology}
Throughout this paper, we follow M. Green's approach to Koszul cohomology
\cite{Gr84a}.
The general setup is the following. Suppose $X$ is
a complex projective variety, $L\in\Pic(X)$ a line bundle, $\mathcal F$
is a coherent sheaf on $X$, and $p,q\in\mathbb Z$. The
canonical contraction map
$$
\wedge^{p+1}H^0(X, L)\otimes H^0(X, L)^\vee\to \wedge^pH^0(X, L),
$$
acting as
$$
(s_0\wedge\ldots\wedge s_p)\otimes \sigma\mapsto
\sum_{i=0}^p(-1)^i\sigma(s_i)(s_0\wedge\ldots\widehat{i}\ldots\wedge s_p),
$$
and the multiplication map
$$
H^0(X, L)\otimes H^0(X, \mathcal F\otimes L^{q-1})\to H^0(X, \mathcal F\otimes L^q),
$$
define together a map
$$
\wedge^{p+1}H^0(X, L)\otimes H^0(X, \mathcal F\otimes L^{q-1})\to
\wedge^pH^0(X, L)\otimes H^0(X, \mathcal F\otimes L^q).
$$

In this way, we obtain a complex (called the {\em Koszul complex})
$$
\wedge^{p+1}H^0(L)\otimes H^0(\mathcal F\otimes L^{q-1})\to
\wedge^pH^0(L)\otimes H^0(\mathcal F\otimes L^q)\to
\wedge^{p-1}H^0(L)\otimes H^0(\mathcal F\otimes L^{q+1}),
$$
whose cohomology at the middle-term is denoted by
$K_{p,q}(X,\mathcal F,L)$. In the particular case
$\mathcal F\cong\mathcal O_X$, to ease the notation,
one drops $\mathcal O_X$ and writes directly $K_{p,q}(X,L)$
for Koszul cohomology.

\medskip

There are some samples of direct applications of Koszul
cohomology:

\begin{ex}
 {\rm If $L$ is ample, then $L$ is normally generated if and only if
$K_{0,q}(X,L)=0$ for all $q\ge 2$. This fact follows directly from
the definition.}
\end{ex}

\begin{ex}
\label{ex: cohomology}
 {\rm If $L$ is globally generated with $H^1(X, L)=0$, then
$$
H^1(X, \mathcal O_X)\cong K_{h^0(L)-2,2}(X,L)
$$
see \cite{Aprodu-Nagel: book}. In particular,
the genus of a curve can be read off (vanishing of) Koszul cohomology
with values in non-special bundles. More generally, under suitable
vanishing assumptions on $L$, all the groups $H^i(X, \mathcal O_X)$ can
be computed in a similar way, and likewise $H^i(X, \mathcal F)$ for
an arbitrary coherent sheaf $\mathcal F$, cf. \cite{Aprodu-Nagel: book}.}
\end{ex}

\begin{ex}
\label{ex: CM}
 {\rm If $L$ is very ample the \emph{Castelnuovo-Mumford regularity}
can be recovered from Koszul cohomology. Specifically,
if $\mathcal F$ is a coherent sheaf on $X$, then
$$
\reg_L(\mathcal F)=\min\{m: K_{p,m+1}(X,\mathcal F,L)=0,\ \mbox{for all}\ p\}.
$$

As a general principle, any invariant that involves multiplication
maps is presumably related to Koszul cohomology.}
\end{ex}

\medskip

Very surprisingly, Koszul cohomology interacts
much closer to the geometry of the variety than might have been expected. This phenomenon
was discovered by Green and Lazarsfeld
\cite[Appendix]{Gr84a}:

\begin{thm}[Green-Lazarsfeld]
 \label{thm: GL nonvan}
Suppose $X$ is a smooth variety and consider a decomposition $L=L_1\otimes L_2$
with $h^0(X, L_i)=r_i+1\ge 2$ for $i\in \{1,2\}$. Then $K_{r_1+r_2-1,1}(X,L)\ne 0$.
\end{thm}

In other words, non-trivial geometry implies non-trivial Koszul cohomology.
We shall discuss the case of curves, which is the most relevant, and
then some consequences in Section 4.

Many problems in the theory of syzygies involve
vanishing/nonvanishing of Koszul cohomology.
One useful definition is the following.

\begin{defn}
 \label{defn: Np}
An ample line bundle $L$ on a projective variety is said to satisfy the
property $(N_p)$ if and only if $K_{i,q}(X,L)=0$ for all $i\le p$ and $q\ge 2$.
\end{defn}

From the geometric view-point, the property $(N_p)$ means that
$L$ is normally generated, the ideal of $X$
in the corresponding embedding is generated by quadrics, and
all the syzygies up to order $p$ are linear. In many cases,
for example canonical curves, or $2$-regular varieties, the
property $(N_p)$ reduces to the single condition $K_{p,2}(X.L)=0$, see e.g.
\cite{Ein}.
This phenomenon justifies the study of various loci given
by the nonvanishing of $K_{p,2}$, see Section 4.

\subsection{Kernel bundles}
The proofs of the facts discussed in
examples \ref{ex: cohomology} and \ref{ex: CM}
use the kernel bundles description which is due
to Lazarsfeld \cite{La1}:
Consider $L$ a globally generated line bundle
on the projective variety $X$, and set
$$
M_L:=\mathrm{Ker}\left(H^0(X, L)\otimes\mathcal O_X
\stackrel{\mathrm{ev}}{\longrightarrow}L\right).
$$
Note that $M_L$ is a vector bundle on $X$ of rank $h^0(L)-1$. For
any coherent sheaf $\mathcal F$ on $X$, and
integer numbers $p\ge 0$, and $q\in \mathbb Z$, we have a
short exact sequence on $X$
$$
0\to \wedge^{p+1}M_L\otimes \mathcal F\otimes L^{q-1}
\to \wedge^{p+1}H^0(L)\otimes \mathcal F\otimes L^{q-1}\to
\wedge^{p}M_L\otimes \mathcal F\otimes L^{q}\to 0.
$$

Taking global sections, we remark that the Koszul
differential factors through the map
$$
\wedge^{p+1}H^0(X, L)\otimes H^0(X, \mathcal F\otimes L^{q-1})
\to H^0(X, \wedge^p M_L\otimes \mathcal F\otimes L^q),
$$
hence we have the following characterization of
Koszul cohomology, \cite{La1}:

\begin{thm}[Lazarsfeld]
\label{prop: ker}
Notation as above. We have
\begin{eqnarray*}
K_{p,q}(X,\mathcal F,L) \cong
\coker(\wedge^{p+1}H^0(L)\otimes H^0(\mathcal F\otimes L^{q-1})
\to H^0(\wedge^p M_L\otimes \mathcal F\otimes L^q)) \\
\cong
\ker(H^1(X,\wedge^{p+1}M_L\otimes\mathcal F\otimes L^{q-1})
\to\wedge^{p+1}H^0(L)\otimes H^1(\mathcal F\otimes L^{q-1})).
\end{eqnarray*}
\end{thm}

Theorem \ref{prop: ker} has some nice direct consequences. The first
one is a duality Theorem which was proved in \cite{Gr84a}.

\begin{thm}
\label{thm: duality}
Let $L$ be a globally generated line bundle on a smooth projective variety $X$ of
dimension $n$. Set $r:= \dim|L|$. If
$$
H^i(X,L^{q-i}) = H^i(X,L^{q-i+1}) = 0,\ \ i=1,\ldots,n-1
$$
then
$$
K_{p,q}(X,L)^{\vee}\cong K_{r-n-p,n+1-q}(X,K_X,L).
$$
\end{thm}

Another consequence of Theorem \ref{prop: ker}, stated
implicitly in \cite{Fa06b} without proof,
is the following:

\begin{thm}
 \label{thm: nonvan}
Let $L$ be a non-special globally generated line bundle on
a smooth curve $C$ of genus $g\ge 2$. Set $d  =\deg(L)$, $r=h^0(C,L)-1$, and
consider $1\le p\le r$. Then
$$
\dim\ K_{p,1}(C,L)-\dim\ K_{p-1,2}(C,L)=p\cdot{d-g\choose p}
\left(\frac{d+1-g}{p+1}-\frac{d}{d-g}\right).
$$
In particular, if
$$
p<\frac{(d+1-g)(d-g)}{d}-1,
$$
then $K_{p,1}(C,L)\ne 0$, and if
$$
\frac{(d+1-g)(d-g)}{d}\le p\le d-g,
$$
then $K_{p-1,2}(C,L)\ne 0$.
\end{thm}

\proof
Since we work with a spanned line bundle on a curve,
Theorem \ref{thm: duality} applies, hence we have
$$
K_{p,1}(C,L)\cong K_{r-p-1,1}(C,K_C,L)^\vee
$$
and
$$
K_{p-1,2}(C,L)\cong K_{r-p,0}(C,K_C,L)^\vee.
$$

Set, as usual, $M_L=\mbox{Ker} \{H^0(C, L)\otimes \mathcal O_C\to L \}$, and
consider the Koszul complex
\begin{equation}
\label{eqn: KC}
0\to\wedge^{r-p}H^0(C, L)\otimes H^0(C, K_C)\to H^0(C, \wedge ^{r-p-1}M_L\otimes K_C\otimes L)\to 0.
\end{equation}

Since $L$ is non-special, $K_{r-p,0}(C,K_C,L)$ is isomorphic
to the kernel of the differential appearing in (\ref{eqn: KC}),
hence the difference which we wish to compute coincides with
the Euler characteristic of the complex (\ref{eqn: KC}).

\medskip

Next we determine $h^0(C, \wedge ^{r-p-1}M_L\otimes K_C\otimes L)$.
Note that $\rk(M_L)=r$ and $\wedge ^{r-p-1}M_L\otimes L\cong \wedge^{p+1}M_L^\vee$.
In particular, since $H^0(C, \wedge^{p+1}M_L)\cong K_{p+1,0}(C,L)=0$, we
obtain
$$
h^0(C, \wedge ^{r-p-1}M_L\otimes K_C\otimes L)=-\chi(\wedge^{p+1}M_L).
$$
Observe that
$$
\deg(\wedge^{p+1}M_L)=\deg(M_L){r-1\choose p}=-d  {r-1\choose p}
$$
and
$$
\rk(\wedge^{p+1}M_L)={r\choose p+1}
$$

From the Riemann-Roch Theorem it follows
$$
-\chi(C, \wedge^{p+1}M_L)=d  {r-1\choose p}+(g-1){r\choose p+1}
$$
and hence
$$
\dim \  K_{p,1}(C,L)-\dim \ K_{p-1,2}(C,L)=d  {r-1\choose p}+(g-1){r\choose p+1}
-g{r+1\choose p+1}.
$$

The formula is obtained by replacing $r=d-g$.
\endproof

\begin{rmk}
\label{rmk: nonvan}
 {\rm
A full version of Theorem \ref{thm: nonvan} for special line
bundles can be obtained in a similar manner
by adding alternating sums of other groups $K_{p-i.i+1}$.
For example, if $L^{\otimes 2}$ is non-special, then
from the complex
$$
0\to\wedge^{r-p+1}H^0(C, L)\otimes H^0(C, K_C\otimes L^{-1})
\to\wedge^{r-p}H^0(C, L)\otimes H^0(C, K_C)\to
$$
$$
\to H^0(C, \wedge ^{r-p-1}M_L
\otimes K_C\otimes L)\to 0
$$
we obtain the following expression for
$$
\dim\ K_{p,1}(C,L)-\dim\ K_{p-1,2}(C,L)+
\dim\ K_{p-2,3}(C,L)=
$$
$$
=d  {r-1\choose p}+(g-1){r\choose p+1}
-g{r+1\choose p+1}+{r+1\choose p}(r-d+g).
$$}
\end{rmk}

\subsection{Hilbert schemes}

Suppose $X$ is a smooth variety, and consider $L\in\Pic(X)$.
A novel description of the
Koszul cohomology of $X$ with values in $L$ was provided
in \cite{Voisin: even} via the Hilbert scheme of points on $X$.

Denote by $X^{[n]}$ the Hilbert scheme parameterizing
zero-dimensional length $n$ subschemes of $X$, let
$X^{[n]}_{\rm curv}$ be the open subscheme
parameterizing curvilinear length $n$ subschemes, and let
$$
\Xi_n\subset X^{[n]}_{\rm curv}\times X
$$
be the incidence subscheme.
For a line bundle $L$ on $X$, the sheaf
$L^{[n]} := q_*p^*L$
is locally free of rank $n$ on $X^{[n]}_{\rm curv}$,
and the fiber over $\xi\in X^{[n]}_{\rm curv}$ is
isomorphic to $H^0(\xi,L\otimes\mathcal O_{\xi})$.

\medskip
There is a natural map
$$
H^0(X,L)\otimes\mathcal O_{X^{[n]}_{\rm curv}}\to L^{[n]},
$$
acting on the fiber over $\xi\in X^{[n]}_{\rm curv}$,
by $s\mapsto s|_{\xi}$.
In \cite{Voisin: even} and
\cite{EGL} it is shown that, by taking wedge powers and global sections,
this map induces an isomorphism:
$$
\wedge^n H^0(X,L)\cong H^0(X^{[n]}_{\rm curv},\det\ L^{[n]}).
$$

Voisin proves that there is an injective map
$$
H^0(\Xi_{p+1},\det\ L^{[p+1]}\boxtimes L^{q-1})
\rightarrow\wedge^p H^0(X,L)\otimes H^0(X,L^q)
$$
whose image is isomorphic to the kernel of the Koszul differential.
This eventually leads to the following result:

\begin{thm}[Voisin \cite{Voisin: even}]
\label{thm: Voisin description}
For all integers $p$ and $q$, the Koszul cohomology
$K_{p,q}(X,L)$ is isomorphic to the cokernel of the
restriction map
$$
H^0(X^{[p+1]}_{\rm curv}\times X,\det\ L^{[p+1]}\boxtimes L^{q-1})
\to
H^0(\Xi_{p+1},\det\ L^{[p+1]}\boxtimes L^{q-1}|_{\Xi_{p+1}}).
$$
In particular,
$$
K_{p,1}(X,L)\cong\coker(H^0(X^{[p+1]}_{\rm curv},\det\ L^{[p+1]})
\stackrel{q^*}{\to}H^0(\Xi_{p+1},q^*\det\ L^{[p+1]}|_{\Xi_{p+1}})).
$$
\end{thm}

\medskip
\begin{rmk}
{\rm
The group $K_{p,q}(X,\mathcal F,L)$ is obtained by replacing $L^{q-1}$ by
$\mathcal F\otimes L^{q-1}$ in the statement of Theorem \ref{thm: Voisin description}.
}
\end{rmk}

The main application of this approach is the proof of the
generic Green Conjecture \cite{Voisin: even}, see Subsection 4.1 for a more
detailed discussion on the subject. The precise statement
is the following.

\begin{thm}[Voisin \cite{Voisin: even} and \cite{Voisin: odd}]
\label{thm: Voisin}
Consider a smooth projective $K3$ surface $S$, such
that $\Pic(S)$ is isomorphic to $\mathbb{Z}^2$, and is freely
generated by $L$ and $\mathcal{O}_S(\Delta)$, where $\Delta$ is a
smooth rational curve such that $\mathrm{deg}(L_{\mid \Delta})=2$, and $L$
is a very ample line bundle with $L^2=2g-2, \,g=2k+1$. Then
$K_{k+1,1}\bigl(S,L\otimes\mathcal{O}_S(\Delta)\bigr)=0$ and
\begin{equation}
 \label{eqn: Voisin odd}
 K_{k,1}(S,L)=0.
\end{equation}
\end{thm}

Voisin's result, apart from settling the Generic Green Conjecture, offers the possibility (via the cohomological calculations carried out in \cite{Fa06b}, see also Section 3), to give a  much shorter proof of the Harris-Mumford Theorem \cite{HM} on the Kodaira dimension of $\mm_g$ in the case of odd genus. This proof does not use intersection theory on the stack of admissible coverings at all and is considerably shorter that the original proof. This approach has been described in full detail in \cite{Fa08}.

\section{Geometric cycles on the moduli space}
\label{sec: cycles}

\subsection{Brill-Noether cycles}
We recall a few basic facts from Brill-Noether theory, see \cite{ACGH}
for a general reference.

For a smooth curve $C$ and integers $r, d\geq 0$ one considers the Brill-Noether locus
$$
W^r_d(C):=\{L\in\Pic^d(C): h^0(C, L)\ge r+1\}
$$
as well as the variety of linear series of type $\mathfrak g^r_d$ on $C$, that is,
$$G^r_d(C):=\{(L, V): L\in W^r_d(C), V\in \textbf{G}(r+1, H^0(L))\}.$$
The locus $W^r_d(C)$  is a deteminantal subvariety of $\Pic^d(C)$ of expected dimension
equal to the {\em{Brill-Noether numer}} $\rho(g,r,d)=g-(r+1)(g-d+r)$. According to the  Brill-Noether Theorem
 for a general curve
$[C]\in \cM_g$, both $W^r_d(C)$ and $G^r_d(C)$ are irreducible varieties of dimension
 $$\dim\ W^r_d(C)=\dim\ G^r_d(C)=\rho(g,r,d).$$
In particular, $W^r_d(C)=\emptyset$ when $\rho(g, r, d)<0$. By imposing the condition that a curve carry
a linear series $\mathfrak g^r_d$ when $\rho(g, r, d)<0$, one can define a whole range of geometric subvarieties of $\cM_g$.

We introduce the Deligne-Mumford stack  $\sigma: \mathfrak{G}^r_d\rightarrow \textbf{M}_g$ classifying pairs $[C, l]$ where $[C]\in \cM_g$ and $l=(L, V)\in G^r_d(C)$ is a linear series $\mathfrak g^r_d$, together
with the projection $\sigma[C, l]:=[C]$. The stack $\mathfrak G^r_d$ has a determinantal structure inside a Grassmann bundle over the universal Picard stack $\mathfrak{Pic}_g^d\rightarrow \textbf{M}_g$ . In particular, each irreducible component of $\mathfrak G^r_d$ has dimension at least $3g-3+\rho(g, r, d)$, cf. \cite{AC81b}. We define the \emph{Brill-Noether cycle}
$$
\cM_{g, d}^r:=\sigma_*(\mathfrak{G}^r_d)=\{[C]\in\cM_g: W^r_d(C)\neq \emptyset \},
$$
together with the substack structure induced from the determinantal structure of $\mathfrak G^r_d$ via the morphism $\sigma$. A result of Steffen \cite{St98} guarantees that each irreducible component of $\cM_{g, d}^r$ has dimension at least $3g-3+\rho(g, r, d)$.

When $\rho(g, r, d)=-1$, Steffen's result coupled with the Brill-Noether Theorem implies that the cycle $\cM_{g, d}^r$ is pure of codimension $1$ inside $\cM_g$. One has the following more precise statement due to Eisenbud and Harris
\cite{EH89}:
\begin{thm}
For integers $g, r$ and $d$ such that $\rho(g, r, d)=-1$, the locus $\cM_{g, d}^r$ is an irreducible divisor on $\cM_g$. The class of its compactification $\mm_{g, d}^r$ inside $\mm_g$ is given by the following formula:
 $$
\mm_{g, d}^r\equiv c_{g, d, r}\left((g+3)\lambda-\frac{g+1}{6}\delta_0-
\sum_{i=1}^{[g/2]} i(g-i)\delta_i\right)\in \mathrm{Pic}(\mm_g).
$$
\end{thm}
The constant $c_{g, d, r}$ has a clear intersection-theoretic interpretation using Schubert calculus.  Note that remarkably, the slope of all the  Brill-Noether divisors on $\mm_g$ is independent of $d$ and $r$ and
$$s(\mm_{g, d}^r)=6+\frac{12}{g+1}$$
for all $r, d\geq 1$ satisfying $\rho(g, r, d)=-1$. For genera $g$ such that $g+1$ is composite, one has as many Brill-Noether divisors on $\mm_g$ as ways of non-trivially factoring $g+1$. It is
natural to raise the following:

\begin{problem}
Construct an explicit linear equivalence between various Brill-Noether divisors $\mm_{g, d}^r$ on $\mm_g$ for different integers $r, d\geq 1$ with $\rho(g, r, d)=-1$.
\end{problem}

The simplest case is $g=11$ when there exist two (distinct) Brill-Noether divisors $\mm_{11, 6}^1$ and $\mm_{11, 9}^2$
and $s(\mm_{11, 6}^1)=s(\mm_{11, 9}^2)=7$. These divisors can be understood in terms of Noether-Lefschetz divisors on the moduli space $\overline{\mathcal{F}}_{11}$ of polarized $K3$ surfaces of degree $2g-2=20$. We recall that there exists a rational $\mathbb P^{11}$-fibration $$\phi:\mm_{11}-->\overline{\mathcal{F}}_{11}, \mbox{    } \ \phi[C]:=[S, \OO_S(C)],$$ where $S$ is the unique $K3$ surface containing $C$, see \cite{M94}.
Noting that $\cM_{11, 6}^1=\cM_{11, 14}^5$ it follows that $\cM_{11, 6}^1=\phi_{| \cM_{11}}^*(NL_1)$, where $NL_1$ is the Noether-Lefschetz divisor on $\F_{11}$ of polarized $K3$ surfaces $S$ with Picard lattice $\mathrm{Pic}(S)=\mathbb Z\cdot [\OO_S(1)]+\mathbb Z\cdot [C]$, where $C^2=20$ and $C\cdot c_1(\OO_S(1))=14$. Similarly, by Riemann-Roch, we have an equality of divisors $\cM_{11, 9}^2=\cM_{11, 11}^3$, and then $\cM_{11, 9}^2=\phi_{| \cM_{11}}^*(NL_2)$, with $NL_2$ being the Noether-Lefschetz divisor whose general point corresponds to a quartic surface $S\subset \mathbb P^3$ with $\mathrm{Pic}(S)=\mathbb Z\cdot [\OO_S(1)]+\mathbb Z\cdot [C]$, where $C^2=20$ and $C\cdot c_1(\OO_S(1))=11$. It is not clear whether $NL_1$ and $NL_2$ should be linearly equivalent on $\mathcal{F}_{11}$.

The next interesting case is $g=23$, see \cite{Fa00}: The three (distinct) Brill-Noether
divisors $\mm_{23, 12}^1$, $\mm_{23, 17}^2$ and $\mm_{23, 20}^1$ are multicanonical in the sense that there exist
explicitly known integers $m, m_1, m_2, m_3\in \mathbb Z_{>0}$ and an effective boundary divisor $E\equiv \sum_{i=1}^{11} c_i \delta_i\in \mathrm{Pic}(\mm_{23})$ such that
$$m_1\cdot \mm_{23, 12}^1+E\equiv m_2\cdot \mm_{23, 17}^2+E=m_3\cdot \mm_{23, 20}^3+E\in |mK_{\mm_{23}}|.$$

\begin{question}
For a genus $g$ such that $g+1$ is composite, is there a good geometric description of the stable base locus
$$\textbf{B}\bigl(\mm_g, |\mm_{g, d}^r|\bigr):=\bigcap_{n\geq 0} \mbox{Bs}(\mm_g, |n\mm_{g, d}^r|)$$
of the Brill-Noether linear system? It is clear that $\textbf{B}(\mm_g, |\mm_{g, d}^r|)$ contains important
subvarieties of $\mm_g$ like the hyperelliptic and trigonal locus, cf. \cite{HaMo}.
\end{question}

Of the higher codimension Brill-Noether cycles, the best understood are the $d$-gonal loci
$$\cM_{g, 2}^1\subset \cM_{g, 3}^1\subset \ldots \subset \cM_{g, d}^1\subset \ldots \subset \cM_g.$$
Each stratum $\cM_{g, d}^1$ is an irreducible variety of dimension $2g+2d-5$. The gonality stratification of $\cM_g$,  apart from being essential in the statement of Green's Conjecture, has
often been used for cohomology calculations or for bounding the cohomological dimension and the affine covering number of $\cM_g$.

\subsection{Koszul cycles}
Koszul cohomology behaves like the usual cohomology in many regards.
Notably, it can be computed in families see \cite{BG85}, or the book
\cite{Aprodu-Nagel: book}:

\begin{thm}
 \label{thm: semicontinuity}
Let $f:X\to S$ a flat family of projective varieties, parameterized
by an integral scheme, $L\in\Pic(X/S)$ a line bundle and $p,q\in \mathbb Z$. Then
there exists a coherent sheaf $\mathcal K_{p,q}(X/S,L)$ on $X$
and a nonempty Zariski open subset $U\subset S$ such that for all $s\in U$ one has that
$\mathcal K_{p,q}(X/S,L)\otimes k(s)\cong K_{p,q}(X_s,L_s)$.
\end{thm}

In the statement above, the open set $U$ is precisely determined by the condition
that all $h^i(X_s,L_s)$ are minimal.

\medskip

By Theorem \ref{thm: semicontinuity}, Koszul cohomology can be
used to construct effective determinantal cycles on the moduli
spaces of smooth curves. This works particularly well
for Koszul cohomology of canonical curves, as $h^i$ remain
constant over the whole moduli space. More generally, Koszul
cycles can be defined over the relative Picard stack
over the moduli space. Under stronger
assumptions, the canonically defined determinantal structure
can be given a better description. To this end, one uses
the description provided by Lazarsfeld kernel bundles.

In many cases, for example canonical curves, or $2$-regular varieties, the
property $(N_p)$ reduces to the single condition $K_{p,2}(X.L)=0$, see for instance
\cite{Ein} Proposition 3.
This phenomenon justifies the study of various loci given
by the non-vanishing of $K_{p,2}$. Note however that extending
this determinantal description over the boundary of the moduli stack (especially over the locus of reducible stable curves) poses considerable
technical difficulties, see \cite{Fa06a}, \cite{Fa06b}.
We now describe a general set-up used to compute
Koszul non-vanishing loci over a partial compactification $\widetilde{\cM}_g$ of  the moduli space $\cM_g$ inside $\mm_g$. As usual, if $\textbf{M}$ is a Deligne-Mumford stack, we denote by $\cM$ its associated coarse moduli space.

\medskip

We fix integers $r, d\geq 1$, such that $\rho(g, r, d)=0$ and  denote by $\textbf{M}_g^0\subset \textbf{M}_g$ the open substack
classifying curves $[C]\in \cM_g$ such that $W_{d-1}^r(C)=
\emptyset$ and $W_d^{r+1}(C)= \emptyset$. Since $\rho(g, r+1, d)\leq -2$ and $\rho(g, r, d-1)=-r-1\leq -2$, it follows  from \cite{EH89} that
$\mathrm{codim}(\cM_g-\cM_g^0, \cM_g)\geq 2$. We further denote by
$\Delta_0^0\subset \Delta_0\subset \mm_g$ the locus of nodal curves
$[C_{yq}:=C/y\sim q]$, where $[C]\in \cM_{g-1}$ is a curve that satisfies the
Brill-Noether Theorem and $y, q\in C$ are \emph{arbitrary} points. Finally, $\Delta_1^0\subset \Delta_1\subset \mm_g$ denotes the open substack classifying curves $[C\cup_y E]$, where $[C]\in \cM_{g-1}$ is  Brill-Noether general, $y\in C$ is an arbitrary point and $[E, y]\in \mm_{1, 1}$ is an arbitrary elliptic tail. Note that every Brill-Noether general curve $[C]\in \cM_{g-1}$ satisfies  $$W_{d-1}^{r}(C)=\emptyset,  \
\ W_d^{r+1}(C)=\emptyset\  \mbox{ and } \mbox{ dim }W^r_d(C)=\rho(g-1, r,
d)=r.$$ We set $\pem_g:=\textbf{M}_g^0\cup \Delta_0^0\cup \Delta_1^0 \subset \rem_g$ and we regard it as a partial compactification of $\textbf{M}_g$. Then following \cite{EH86}
we consider the Deligne-Mumford stack
$$\sigma_0:\widetilde{\mathfrak G}^r_d\rightarrow \pem_g$$
classifying pairs $[C, l]$ with $[C]\in \widetilde{\cM}_g$ and $l$ is a limit linear series of type $\mathfrak g^r_d$ on $C$.
We remark that for any curve $[C]\in \cM_{g}^0\cup \Delta_0^0$ and $L\in W^r_d(C)$, we have that
$h^0(C, L)=r+1$ and that $L$ is globally generated. Indeed, for a smooth curve $[C]\in \cM_{g}^0$ it follows that $W_{d}^{r+1}(C)=\emptyset$,
so necessarily $W^r_d(C)=G^r_d(C)$.
For a point $[C_{yq}]\in \Delta_0^0$ we have the identification
$$\sigma_0^{-1}\bigl[C_{yq}\bigr]=\{L\in W^r_d(C): h^0(C,
L\otimes \OO_C(-y-q))=r\},$$
where we note that since the normalization $[C]\in \cM_{g-1}$ is assumed to be Brill-Noether general, any sheaf  $L\in \sigma_0^{-1}[C_{yq}]$ satisfies $$h^0(C, L\otimes \OO_C(-y))=h^0(C, L\otimes \OO_C(-q))=r$$ and $h^0(C, L)=r+1$. Furthermore,
$\overline{W}^r_d(C_{yq})=W^r_d(C_{yq})$, where the left-hand-side
denotes the closure of $W^r_d(C_{yq})$ inside the variety
$\overline{\mbox{Pic}}^d(C_{yq})$ of torsion-free sheaves on
$C_{yq}$. This follows because a non-locally free torsion-free sheaf in
$\overline{W}^r_d(C_{yq})-W^r_d(C_{yq})$ is of the form $\nu_*(A)$,
where $A\in W_{d-1}^r(C)$ and $\nu:C\rightarrow C_{yq}$ is the
normalization map. But we know that $W_{d-1}^r(C)=\emptyset$,
because $[C]\in \cM_{g-1}$ satisfies the Brill-Noether
Theorem. The conclusion of this discussion is that $\sigma: \widetilde{\mathfrak G}^r_d
\to\pem_g$ is proper. Since $\rho(g, r, d)=0$, by general Brill-Noether theory,
there exists a unique irreducible component of $\mathfrak G^r_d$
which maps onto $\textbf{M}_g^0$.

In \cite{Fa06b}, a universal Koszul non-vanishing locus over a partial compactification of the  moduli space of curves is introduced. Precisely, one constructs two locally free sheaves $\mathcal A$ and $\mathcal B$ over $\widetilde{\mathfrak
G}^r_d$ such that for a point $[C, l]$ corresponding to a \emph{smooth} curve $[C]\in \cM_g^0$ and a (necessarily complete and globally generated linear series) $l=(L, H^0(C, L))\in G^r_d(C)$ inducing a map $C\stackrel{|L|}\longrightarrow \mathbb P^r$,
we have the following description of the fibres
$$
\mathcal A(C,L)=H^0\bigl(\mathbb P^r, \wedge^p M_{\mathbb P^r}\otimes \OO_{\mathbb P^r}(2)\bigr) \mbox{
}\mbox{ and }\mathcal B(C,L)=H^0\bigl(C, \wedge^p M_L\otimes L^{\otimes 2}\bigr).
$$
There is a natural vector bundle morphism $\phi:\mathcal A\rightarrow \mathcal B$ given by
restriction. From Grauert's Theorem it follows that both $\mathcal A$ and $\mathcal B$
are vector bundles over $\mathfrak G^r_d$  and from Bott's Theorem (in the case of $\mathcal{A}$) and
Riemann-Roch (in the case of $\cB$) respectively,
we compute their ranks
$$
\mbox{rank}(\mathcal A)=(p+1){r+2 \choose p+2}\mbox{  and  }\mbox{
rank}(\mathcal B)={r\choose p}\Bigl(-\frac{pd}{r}+2d+1-g\Bigr).$$ Note that
$M_L$ is a stable vector bundle (again, one uses that $[C]\in \cM_g^0$), hence  $H^1(C, \wedge^p M_L\otimes L^{\otimes
2})=0$ and then $\mbox{rank}(\mathcal B)=\chi(C, \wedge^p M_L\otimes L^{\otimes 2})$ can be computed from Riemann-Roch.
We have the following result, cf. \cite{Fa06b} Theorem 2.1:

\begin{thm}\label{specsyzvirt}
\label{thm: ni}
The cycle
$$
\mathcal{U}_{g, p}:=\{(C,L)\in \mathfrak G^r_d:
K_{p,2}(C,L)\ne 0\},
$$
is the degeneracy locus of the vector bundle map
$\phi:\mathcal A\rightarrow \mathcal B$ over $\mathfrak G^r_d$.
\end{thm}

Under suitable numerical restrictions, when $\mathrm{rank}(\cA)=\mathrm{rank}(\cB)$, the cycle constructed
in Theorem \ref{thm: ni} is a virtual divisor on $\mathfrak G^r_d$. This happens precisely when
$$
r:=2s+sp+p,\ g:=rs+s \mbox{ and }
d:=rs+r.
$$
for some $p\geq 0$ and $s\geq 1$.  The first remarkable
case occurs when $s=1$. Set $g=2p+3$,
$r=g-1=2p+2$, and $d=2g-2=4p+4$. Note that, since the
canonical bundle is the only $\mathfrak g^{g-1}_{2g-2}$ on a curve of genus $g$, the
Brill-Noether stack is isomorphic to $\textbf{M}_g$.  The notable fact that the cycle
in question is an actual divisor follows directly from Voisin's
Theorem \ref{thm: Voisin}
and from Green's Hyperplane Section Theorem \cite{Voisin: odd}.

Hence $\mathcal{Z}_{g, p}:=\sigma_*(\mathcal{U}_{g, p})$ is a virtual
divisor on $\mathcal M_g$ whenever $$g=s(2s+sp+p+1).$$
In the next section, we explain how to extend the morphism $\phi:\cA\rightarrow \cB$ to a morphism of locally free sheaves over the stack $\widetilde{\mathfrak G}^r_d$ of limit linear series and reproduce the class formula
proved in \cite{Fa06a} for the degeneracy locus of this morphism.

\subsection{Divisors of small slope}

In \cite{Fa06a} it was shown that
the determinantal structure of $\mathcal{Z}_{g, p}$ can be extended
over $\overline{\mathcal M}_g$ in such a way that whenever $s\geq 2$, the resulting
virtual slope violates the Harris-Morrison Slope Conjecture. One has the
following general statement:

\begin{thm}
\label{thm: div}
If $\sigma:\widetilde{\mathfrak G}^r_d\rightarrow \pem_g$ is the
compactification of $\mathfrak G^r_d$ given by limit linear series,
then there exists a natural extension of the vector bundle map
$\phi: \mathcal A\rightarrow \mathcal B$ over $\widetilde{\mathfrak{G}}^r_d$ such
that $\overline{\mathcal{Z}}_{g, p}$ is the image of the degeneracy
locus of $\phi$. The class of the pushforward to $\pem_g$ of the
virtual degeneracy locus of $\phi$ is given by
$$
\sigma_*(c_1(\mathcal G_{p, 2}-\mathcal H_{p, 2}))
\equiv a\lambda-b_0\delta_0-b_1\delta_1-\cdots -b_{[\frac{g}{2}]} \delta_{[\frac{g}{2}]},
$$
where $a, b_0, \ldots, b_{[\frac{g}{2}]}$ are explicitly given coefficients
such that $b_1=12b_0-a$,  $b_i\geq b_0$ for $1\leq i\leq [g/2]$ and
$$
s\bigl(\sigma_*(c_1(\mathcal G_{p, 2}-\mathcal H_{p, 2}))\bigr)=
\frac{a}{b_0}=6\frac{f(s,p)}{(p+2)\ s h(s, p)}, \mbox{ with}
$$
\begin{center}
$f(s,p)=(p^4+24p^2+8p^3+32p+16)s^7+(p^4+4p^3-16p-16)s^6-(p^4+7p^3+13p^2-12)s^5-
(p^4+2p^3+p^2+14p+24)s^4
+(2p^3+2p^2-6p-4)s^3+(p^3+17p^2+50p+41)s^2+(7p^2+18p+9)s+2p+2$
\end{center}
and
\noindent
\begin{center}
$h(s,p)=(p^3+6p^2+12p+8)s^6+(p^3+2p^2-4p-8)s^5-(p^3+7p^2+11p+2)s^4-\newline
-(p^3-5p)s^3+(4p^2+5p+1)s^2+ (p^2+7p+11)s+4p+2.$
\end{center}
Furthermore, we have that $$6<\frac{a}{b_0}<6+\frac{12}{g+1}$$
whenever $s\geq 2$. If the morphism $\phi$ is generically
non-degenerate, then $\overline{\mathcal{Z}}_{g, p}$ is a divisor on
$\overline{\mathcal M}_g$ which gives a counterexample to the Slope Conjecture for
$g=s(2s+sp+p+1)$.
\end{thm}

A few remarks are necessary. In the case $s=1$ and $g=2p+3$, the vector bundles $\cA$ and $\cB$ exist not only over a partial compactification of $\pem_g$ but can be extended
(at least) over the entire stack $\textbf{M}_g\cup \Delta_0$ in such a way that
$\cB(C, \omega_C)=H^0(C, \wedge^p M_{\omega_C}\otimes \omega_C^2)$ for any $[C]\in \cM_g\cup \Delta_0$.
Theorem \ref{thm: div} reads in this case, see also \cite{Fa08} Theorem 5.7:
\begin{equation}\label{hurwitz}
[\overline{\cZ}_{2p+3, p}]^{virt}=c_1(\cB-\cA)=\frac{1}{p+2}{2p\choose p}\Bigl(6(p+3)\lambda-(p+2)\delta_0-6(p+1)\delta_1\Bigr),
\end{equation}
in particular $s([\overline{\cZ}_{2p+3, p}]^{virt})=6+12/(g+1)$.

Particularly interesting is the case $p=0$ when the condition $K_{0, 2}(C, L)=0$ for $[C, L]\in \mathfrak{G}^r_d$, is equivalent to the multiplication map $$\nu_2(L): \mbox{Sym}^2 H^0(C, L)\rightarrow H^0(C, L^{\otimes 2})$$ not being an isomorphism. Note that $\nu_2(L)$ is a linear map between vector spaces of the same dimension and $\cZ_{g, 0}$ is the failure locus of the Maximal Rank Conjecture:
\begin{cor}\label{maxrankk}
For $g=s(2s+1), r=2s, d=2s(s+1)$ the slope of the virtual class of
the locus of those $[C]\in \mm_g$ for which there exists $L\in
W^r_d(C)$ such that the embedded curve $C\stackrel{|L|}\hookrightarrow \PP^r$ sits on a
quadric hypersurface, is equal to
$$s(\overline{\mathcal{Z}}_{s(2s+1),
0})=\frac{3(16s^7-16s^6+12s^5-24s^4-4s^3+41s^2+9s+2)}{s
(8s^6-8s^5-2s^4+s^2+11s+2)}.$$
\end{cor}

\section{Conjectures on Koszul cohomology of curves}

\subsection{Green's Conjecture}

In what follows, we consider $(C,K_C)$ a
smooth canonical curve of genus $g\ge 2$. In this case, the
Duality Theorem \ref{thm: duality} applies,
and the distribution of the numbers
$b_{p,q}:=\dim\ K_{p,q}(C,K_C)$ organized in a table (the {\em
Betti table}) is the following:

\[
\left[
\begin{array}{cccccc}
  b_{0,0}	& 0		 &   0  	 & \ldots & 0  		 & 0\\
  0 		& b_{1,1} & b_{2,1} & \ldots & b_{g-3,1}& b_{g-2,1}\\
  b_{0,2}	& b_{1,2} & b_{2,2} & \ldots & b_{g-3,2}& 0\\
  0		&   0  	 & 0 		 & \ldots & 0 		 & b_{g-2,3}
\end{array}
\right]
\]
The Betti table is symmetric with respect to its center, that is,
$b_{i, j}=b_{g-2-i, 3-j}$
and all the other entries not marked here are zero.

Trying to apply the Non-Vanishing Theorem \ref{thm: GL nonvan}
to the canonical bundle $K_C$, we obtain one
{\em condition} and one {\em quantity}. The condition
comes from the hypothesis that for a decomposition
$K_C=L_1\otimes (K_C\otimes L_1^{\vee})$, Theorem \ref{thm: GL nonvan}
is applicable
whenever
\begin{equation}
\label{eqn: Cliff contrib}
r_1+1:=h^0(C, L_1)\ge 2\mbox{ and }r_2+1:=h^1(C, L_1)\ge 2.
\end{equation}
A line bundle $L_1$ satisfying (\ref{eqn: Cliff contrib}) is said
to {\em contribute to the Clifford index of $C$}.

The quantity that appears in Theorem \ref{thm: GL nonvan} is
the {\em Clifford index} itself. More precisely
$$
r_1+r_2-1=g-\Cliff(L_1)-2,
$$
where
$$
\Cliff(L_1):=\deg(L_1)-2h^0(L_1)+2.
$$
Clifford's Theorem \cite{ACGH} says that $\Cliff(L_1)\ge 0$,
and $\Cliff(L_1)>0$ unless $L_1$ is a $\mathfrak{g}^1_2$.
Following \cite{GL86} we define the {\em Clifford index of $C$} as
the quantity
$$
\Cliff(C):=\mbox{min}\{\Cliff(L_1):\ L_1\mbox{ contributes to the
Clifford index of } C\}.
$$

In general, the Clifford
index will be computed by minimal pencils.
Specifically, a general $d$-gonal curve $[C]\in \cM_{g, d}^1$ (recall
that the gonality strata are irreducible) will have $\Cliff(C)=d-2$.
However, this equality is not valid for all curves, that is,
there exist curves $[C]\in \cM_g$ with  $\Cliff(C)<\gon(C)-2$, basic
examples being plane curves, or exceptional curves
on $K3$ surfaces. Even in these exotic cases, Coppens and
Martens \cite{CM} established the precise relation $\Cliff(C)=\gon(C)-3$.

\medskip

Theorem \ref{thm: GL nonvan} implies the following non-vanishing
$$
K_{g-\mathrm{Cliff}(C)-2,1}(C,K_C)\ne 0,
$$
and Green's Conjecture predicts optimality of Theorem \ref{thm: GL nonvan}
for canonical curves:

\begin{con}
\label{conj: Green}
For any curve $[C]\in \cM_g$ we have the vanishing $K_{p,1}(C,K_C)=0$ for all
$p\ge g-\mathrm{Cliff}(C)-1$.
\end{con}

In the statement of Green's Conjecture, it suffices
to prove the vanishing of $K_{g-\mathrm{Cliff}(C)-1,1}(C,K_C)$ or,
by duality, that $K_{\mathrm{Cliff}(C)-1,2}(C,K_C)=0$.

\vskip 3pt

We shall analyze some basic cases:

\begin{ex}
\label{ex: Noether}
{\rm Looking at $K_{0,2}(C,K_C)$, Green's Conjecture
predicts that it is zero for all non-hyperelliptic curves.
Or, the vanishing of $K_{0,2}(C, K_C)$ is equivalent to
the projective normality of the canonical curve. This
is precisely the content of the classical Max Noether Theorem \cite{ACGH}, p. 117.}
\end{ex}

\begin{ex}
\label{ex: Enriques-Petri}
{\rm For a non-hyperelliptic curve,
we know that $K_{1,2}(C,K_C)=0$ if and only
if the canonical curve $C\subset \mathbb P^{g-1}$
is cut out by quadrics. Green's Conjecture predicts that
$K_{1,2}(C,K_C)=0$ unless the curve is hyperelliptic, trigonal
or a smooth plane quintic. This is precisely the Enriques-Babbage-Petri
Theorem, see \cite{ACGH}, p. 124.}
\end{ex}

Thus Conjecture \ref{conj: Green} appears as  a sweeping generalization of two famous classical theorems. Apart from these classical results, strong evidence has been found for Green's Conjecture
(and one should immediately add, that not a shred of evidence has been found suggesting that the conjecture
might fail even for a single curve $[C]\in \cM_g$).
For instance, the conjecture is true
for general curves in \emph{ any} gonality stratum $\cM_{g, d}^1$,
see \cite{Ap05}, \cite{Teixidor02} and \cite{Voisin: even}. The proof of this
fact relies on semi-continuity. Since
 $\cM_{g, d}^1$ is irreducible,  it suffices to find one example
of a $d$-gonal curve that satisfies the conjecture,
for any $2\le d\le (g+2)/2$; here we also need the
fact mentioned above, that the Clifford index
of a general $d$-gonal curve is $d-2$. The most
important and challenging case, solved
by Voisin  \cite{Voisin: even}, was the case of curves of odd genus $g=2d-1$
and maximal gonality $d+1$.
Following Hirschowitz and Ramanan \cite{HR98} one can compare the Brill-Noether  divisor
$\cM_{g, d}^1$ of curves  with a $\mathfrak g^1_d$ and the virtual divisor of curves $[C]\in \cM_{g}$
with $K_{d-1, 1}(C, K_C)\neq 0$. The non-vanishing Theorem \ref{thm: GL nonvan}
gives a set-theoretic inclusion $\mathcal{M}^1_{g,d}\subset\mathcal{Z}_{g,d-2}$.
Now, we compare the class $[\mathcal{Z}_{g, d-2}]^{virt}\in \mbox{Pic}(\cM_g)$ of the virtual divisor
$\mathcal{Z}_{g,d-2}$ to the class
$[\mathcal{M}^1_{g,d}]$ computed in \cite{HM}. One
finds the following relation $$[\cZ_{g, d-2}]^{virt}=(d-1)[\cM_{g, d}^1]\in \mathrm{Pic}(\cM_g)$$ cf. \cite{HR98}; Theorem \ref{thm: div} in the particular case $s=1$
provides an extension of this equality to a partial compactification of $\cM_g$.
Green's Conjecture for general curves of odd genus \cite{Voisin: odd} implies
that $\mathcal{Z}_{g,d-2}$ is a genuine divisor on $\cM_g$. Since a general curve $[C]\in \cM_{g, d}^1$  satisfies
$$\dim\ K_{d-1,1}(C,K_C)\ge d-1,$$ cf. \cite{HR98}, one finds the set-theoretic equality
$\mathcal{M}^1_{g,d} = \mathcal{Z}_{g,d-2}$. In particular
we obtain the following strong characterization of curves of odd genus and maximal gonality:

\begin{thm}[Hirschowitz-Ramanan, Voisin]
\label{thm: HRV}
If $C$ is a smooth curve of genus $g=2d-1\geq 7$, then
$K_{d-1,1}(C,K_C)\ne 0$ if and only if $C$ carries a $\mathfrak{g}^1_d$.
\end{thm}

Voisin proved Theorem \ref{thm: Voisin}, using Hilbert
scheme techniques, then she applied Green's Hyperplane Section
Theorem  \cite{Gr84a} to obtain the desired example
of a curve $[C]\in \cM_g$ satisfying Green's Conjecture.

\medskip

Starting from Theorem \ref{thm: HRV}, all the other generic
$d$-gonal cases are obtained in the following refined form, see \cite{Ap05}:

\begin{thm}
\label{thm: small BN}
We fix integers $g$ and $d\geq 2$ such that $2\leq d\leq [g/2]+1$.
For any smooth curve $[C]\in \cM_g$  satisfying the condition
\begin{equation}
\label{eqn: small BN}
\dim \ W^1_{g-d+2}(C)\le g-2d+2=\rho(g, 1, g-d+2),
\end{equation}
we have that
$
K_{g-d+1,1}(C,K_C)=0.
$
In particular, $C$ satisfies Green's Conjecture.\end{thm}

Note that the condition $d\leq [g/2]+1$  excludes
the case already covered by Theorem \ref{thm: HRV}.
The proof of Theorem \ref{thm: small BN} relies on constructing
a singular stable curve $[C']\in \mm_{2g+3-2d}$ of maximal gonality $g+3-d$ (that is, $[C']\notin \mm_{2g+3-d, g+2-d}^1$), starting from  any smooth curve $[C]\in \cM_g$ satisfying (\ref{eqn: small BN}). The curve $C'$ is obtained from $C$ by gluing together
$g+3-2d$ pairs of general points of $C$,
and then applying an analogue of Theorem \ref{thm: HRV} for
singular stable curves, \cite{Ap05}, see Section \ref{sec: gonality conj}.
The version in question is the following, cf. \cite{Ap05} Proposition 7.
The proof we give here is however slightly different:

\begin{thm}
\label{thm: nodal}
For any nodal curve $[C']\in \cM_{g'}\cup \Delta_0$, with
$g'=2d'-1\ge 7$ such that
$K_{d'-1,1}(C',\omega_{C'})\ne 0$, it follows that $[C'] \in\overline{\mathcal{M}}^1_{g',d'}$.
\end{thm}

\proof
By duality, we obtain the following equality of cycles on $\widetilde{\cM}_{g'}$:
$$
\{[C']:\ K_{d'-1,1}(C',\omega_{C'})\ne 0\}
=\{[C']:\ K_{d'-2,2}(C',\omega_{C'})\ne 0\}=:\overline{\cZ}_{g', d'-2}.
$$

Theorem \ref{thm: div} shows that this locus is a virtual
divisor on $\widetilde{\cM}_{g'}$ whose class is given by formula (\ref{hurwitz}) and
Theorem \ref{thm: Voisin} implies that $\overline{\cZ}_{g', d'-2}$ is actually a divisor. Comparing its class
against the class of the  Hurwitz divisor $\overline{\mathcal{M}}^1_{g',d'}$  \cite{HM}, we find that
$$
\overline{\cZ}_{g', d'-2}\equiv (d'-1)\mm_{g', d'}^1\in \mathrm{Pic}(\pem_{g'}).
$$
Note that this is a stronger statement than the one \cite{HR98} Proposition 3.1, being an equality of codimension $1$ cycles on the compactified moduli space $\widetilde{\cM}_{g'}$,
rather than on $\cM_{g'}$. The desired statement follows immediately since for any curve $[C']\in \cM_{g', d'}^1$ one has
$\mbox{dim } K_{d'-1, 1}(C', \omega_{C'})\geq d'-1$, hence the degeneracy locus $\overline{\cZ}_{g', d'-2}$ contains $\mm_{g', d'}^1$ with multiplicity at least $d'-1$.
\endproof

We return to the discussion on
Theorem \ref{thm: small BN}
(the proof will be resumed in the next
subsection). By duality, the
vanishing in the statement above can be rephrased as
$$
K_{d-3,2}(C,K_C)=0.
$$

The condition (\ref{eqn: small BN}) is equivalent to
a string of inequalities $$\dim\ W^1_{d+n}(C) \le n$$ for all $0\le n\le g-2d+2$,
in particular $\gon(C)\ge d$.
This condition is satisfied for a general $d$-gonal curve, cf. \cite{Ap05}. More generally, if $[C]\in \cM_{g, d}^1$
is a general $d$-gonal curve then any irreducible component $$Z\neq W^1_d(C)+W_{n-d}(C)$$ of $W^1_n(C)$ has dimension $\rho(g, 1, n)$. In particular, for $\rho(g, 1, n)<0$ it follows that $W^1_n(C)=W^1_{d}(C)+W_{n-d}(C)$ which of course implies (\ref{eqn: small BN}). For $g=2d-2$, the inequality (\ref{eqn: small BN})  becomes necessarily an
equality and it reads: the curve $C$ carries finitely many $\mathfrak{g}^1_d$'s
of minimal degree.

\medskip

We make some comments regarding condition
(\ref{eqn: small BN}). Let us suppose that $C$ is
non-hyperelliptic and $d\ge 3$.
From  Martens's Theorem
\cite{ACGH} p.191, it follows that
$\dim \ W^1_{g-d+2}(C)\le g-d-1$. Condition (\ref{eqn: small BN})
requires the better bound $g-2d+2\le g-d-1$. However,
for $d=3$, the two bounds are the same, and Theorem
\ref{thm: small BN} shows that $K_{0,0}(C,K_C)=0$,
for any non-hyperelliptic curve, which is Max Noether's
Theorem, see also Example \ref{ex: Noether}. Applying
Mumford's Theorem \cite{ACGH} p.193, we obtain the better bound
$\dim \ W^1_{g-d+2}(C)\le g-d-2$ for $d\ge 4$, unless
the curve is trigonal, a smooth plane quintic or a double
covering of an elliptic curve. Therefore, if $C$ is not
one of the three types listed above, then
$K_{1,2}(C,K_C)=0$, and we recover the
of Enriques-Babbage-Petri Theorem, see also Example \ref{ex: Enriques-Petri}
(note, however, the exception made to bielliptic curves).

Keem has improved the dimension bounds
for $W^1_{g-d+2}(C)$. For $d\ge 5$ and $C$  a curve that has neither a $\mathfrak g^1_4$ nor is
a smooth plane sextic, one has the inequality
$\dim\ W^1_{g-d+2}(C)\le g-d-3$, cf. \cite{Ke90}
Theorems 2.1 and 2.3.
Consequently, Theorem \ref{thm: small BN}
implies the following result which is a complete solution to Green's
Conjecture for $5$-gonal curves:
\begin{thm}[Voisin \cite{Voisin: tetragonales}, Schreyer \cite{Sc3}]
\label{thm: Voisin-Schreyer}
If $K_{2,2}(C,K_C)\ne 0$, then $C$ is
hyperelliptic, trigonal, tetragonal or a smooth plane sextic, that is, $\mathrm{Cliff}(C)\leq 2$.
\end{thm}

Geometrically, the vanishing of $K_{2,2}(C,K_C)$ is
equivalent to the ideal of the canonical curve being
generated by quadrics, and the minimal relations among the generators
being linear.

Theorem 3.1 from \cite{Ke90}  gives the next
bound $\dim \ W^1_{g-d+2}(C)\le g-d-4$, for $d\ge 6$ and
$C$ with $\gon(C)\ge 6$ which does not admit a covering of degree
two or three on another curve, and which is not a plane curve.
The following improvement of Theorem \ref{thm: Voisin-Schreyer} is then
obtained directly from Theorem \ref{thm: small BN} and
\cite{Ap05} Theorem~3.1:

\begin{thm}
\label{thm: Green hexagonal}
If $g\ge 12$ and $K_{3,2}(C,K_C)\ne 0$, then
$C$ is one of the following: hyperelliptic, trigonal, tetragonal,
pentagonal, double cover of an genus $3$ curve,
triple cover of an elliptic curve, smooth plane septic. In other words, if $\mathrm{Cliff}(C)\geq 4$ then
$K_{3, 2}(C, K_C)=0$.
\end{thm}

Theorem \ref{thm: Green hexagonal} represents
the solution to Green's Conjecture for hexagonal curves.
Likewise, Theorem \ref{thm: small BN} can be used
together with the Brill-Noether theory to prove
Green's Conjecture for any gonality $d$ and large genus.
The idea is to apply Coppens' results \cite{Co83}.

\begin{thm}
\label{thm: Green large genus}
If $g\ge 10$ and $d\ge 5$ are two integers such that
$g>(d-2)(2d-7)$, and $C$ is any $d$-gonal curve
of genus $g$ which does not admit any morphism
of degree less than $d$ onto another different smooth curve,
then $\Cliff(C)=d-2$ and Green's Conjecture is verified for $C$, i.e.
$K_{g-d+1,1}(C,K_C)=0$.
\end{thm}

The statement
of  Conjecture \ref{conj: Green} (meant as a vanishing result)
for hyperelliptic curve is empty, hence the
interesting cases begin with $d\ge 3$.

\medskip

It remains to verify Green's Conjecture for curves
which do not verify (\ref{eqn: small BN}). One result in this direction was proved in
\cite{ApP06}.
\begin{thm}
\label{thm: Aprodu-Pacienza}
Let $S$ be a $K3$ surface with
$\mathrm{Pic}(S)=\mathbb{Z}\cdot H \oplus{Z}\cdot \ell$,
with $H$ very ample, $H^2=2r - 2\ge 4$, and
$H\cdot \ell=1$. Then any smooth curve in the linear system
$|2H+\ell|$ verifies Green's conjecture.
\end{thm}

Smooth curves in the linear system
$|2H+\ell|$ count among the few known
examples of curves whose Clifford index is not computed by pencils, i.e.
$\mathrm{Cliff}(C)= \mathrm{gon}(C) - 3$, \cite{ELMS}
(other obvious examples are plane
curves, for which Green's Conjecture was verified before, cf. \cite{Lo}). Such curves are the
most special ones in the moduli space of curves from the point of view of the Clifford
dimension. Hence, this case may be considered as opposite to that of a general
curve of fixed gonality. Note that these  curves carry an one-parameter
family of pencils of minimal degree, hence the condition (\ref{eqn: small BN})
is not satisfied.

\subsection{The Gonality Conjecture.}
\label{sec: gonality conj}

The Green-Lazarsfeld Gonality Conjecture \cite{GL86}
predicts that the
gonality of a curve can be read off the Koszul
cohomology with values in any sufficiently positive
line bundle.

\begin{con}[Green-Lazarsfeld]
 \label{conj: gonality}
Let $C$ be a smooth curve of gonality $d$, and $L$ a
sufficiently positive line bundle
on $C$. Then
$$
K_{h^0(L)-d,1}(C,L)=0.
$$
\end{con}

Theorem \ref{thm: nonvan} applied to $L$ written as a sum of a minimal
pencil and the residual bundle yields to
$$
K_{h^0(L)-d-1,1}(C,L)\not=0.
$$

Note that if $L$ is sufficiently positive, then the Green-Lazarsfeld
Nonvanishing Theorem is optimal when applied for a decomposition where
one factor is a pencil. Indeed, consider any decomposition $L=L_1\otimes L_2$
 with $r_1=h^0(C,L_1)-1\ge 2$,
and $r_2=h^0(C,L_1)-2\ge 2$. Since $L$ is sufficiently positive,
the linear system $|K_C^{\otimes 2}\otimes L^\vee|$
is empty, and finiteness of the addition map of divisors shows
that at least one of the two linear systems $|K_C\otimes L_i^\vee|$
is empty.  Suppose $|K_C\otimes L_2^\vee|=\emptyset$, choose a point $x\in C-\mbox{Bs}(|L_1|)$ and consider a new decomposition
$L=L_1'\otimes L_2'$, with $L_1'=L_1\otimes \mathcal{O}_C(-x)$,
and $L_2'=L_2\otimes \mathcal{O}_C(x)$. Denoting
as usual $r'_i=h^0(C,L_i)-1$, we find that
$r'_1+r'_2-1 = r_1+r_2-1$, whereas $r'_1=r_1-1$, and $L_2'$
is again non-special. We can apply an inductive argument
until $r_1$ becomes $1$.
Hence the Gonality Conjecture predicts  the optimality of the
Green-Lazarsfeld Nonvanishing Theorem.
However, one major disadvantage of this
statement is that ``sufficiently positive'' is
not a precise condition. It was proved in \cite{Ap1} that by
adding effective divisors to bundles that
verify the Gonality Conjecture we obtain again bundles
that verify the conjecture. Hence, in order
to find a precise statement for Conjecture \ref{conj: gonality}
one has to study the edge cases.

In most generic cases (general curves in gonality
strata, to be more precise), the Gonality Conjecture
can be verified for line bundles of degree $2g$,
see \cite{AV} and \cite{Ap05}. The test bundles are obtained
by adding two generic points to the canonical bundle.

\begin{thm}[\cite{Ap05}]
\label{thm: GL small BN}
For any $d$-gonal curve $[C]\in \cM_g$  with $d\leq [g/2]+2$
which satisfies the condition (\ref{eqn: small BN}),
and for  general points $x, y\in C$, we have that
$K_{g-d+1,1}\bigl(C,K_C\otimes \OO_C(x+y)\bigr)=0$.
\end{thm}

The case not covered by Theorem \ref{thm: GL small BN} is slightly different.
A general curve $C$ of odd genus carries infinitely many
minimal pencils, hence a bundle of type $K_C\otimes \OO_C(x+y)$ can never
verify the vanishing predicted by the Gonality Conjecture.
Indeed, for any two points $x$ and $y$ there exists a
minimal pencil $L_1$ such that $H^0(C,L_1(-x-y))\ne 0$,
and we apply Theorem \ref{thm: GL nonvan}. However, adding
{\em three} points to the canonical bundle solves the problem, cf. \cite{Ap04},
\cite{Ap05}.

\begin{thm}
\label{thm: GL odd genus}
For any curve $[C]\in \cM_{2d-1}$  of
maximal gonality $\mathrm{gon}(C)=d+1$ and for general points
$x, y, z\in C$, we have that
$K_{d,1}\bigl(C,K_C\otimes \OO_C(x+y+z)\bigr)=0$.
\end{thm}

\medskip

The proofs of Theorems \ref{thm: small BN},
\ref{thm: GL small BN} and \ref{thm: GL odd genus}
are all based on the same idea. We start with a smooth curve $C$ and construct
a stable curve of higher genus out of it, in such a way  that
the Koszul cohomology does not change. Then we apply
a version of Theorem \ref{thm: HRV} for singular curves.

\vskip 3pt
\noindent
\emph{Proof of Theorems \ref{thm: small BN} and \ref{thm: GL small BN}.}
We start with $[C]\in \cM_g$ satisfying the condition (\ref{eqn: small BN}).
We claim that if we choose $\delta:=g+3-2d$ pairs of general points $x_i, y_i\in C$
for $1\leq i\leq \delta$, then the resulting stable curve $$\bigl[C':=\frac{C}{x_1\sim y_1, \ldots, x_{\delta}\sim y_{\delta}}\bigr]\in \mm_{2g+3-2d}$$  is a curve of maximal gonality, that is,  $g+3-d$.
Indeed, otherwise $[C']\in \mm_{2g+3-2d, g+2-d}^1$ and this implies that there exists a
degree $g+2-d$ admissible covering $f:\tilde{C}\rightarrow R$ from a nodal curve $\tilde{C}$ that is semi-stably equivalent to $C'$, onto a genus $0$ curve $R$. The curve $C$ is a subcurve of $\tilde{C}$ and
if $\mbox{deg}(f_{| C})=n\leq g+2-d$, then it follows that $f_{| C}$ induces a pencil $\mathfrak g^1_n$ such that $f_{| C}(x_i)=f_{| C}(y_i)$ for $1\leq i\leq \delta$. Since the points $x_i, y_i\in C$ are general, this implies that $\mbox{dim } W^1_n(C)+\delta \geq 2\delta$, which contradicts (\ref{eqn: small BN}).

To conclude, apply Theorem \ref{thm: nodal} and use the following inclusions,
\cite{Voisin: even}, \cite{AV}:
$$K_{g-d+1,1}(C,K_C)\subset K_{g-d+1,1}(C,K_C(x+y))\subset
K_{g-d+1,1}(C',\omega_{C'}).$$
\hfill
$\Box$

\begin{rmk}{\rm
The proofs of Theorems \ref{thm: small BN}, \ref{thm: nodal} and \ref{thm: GL small BN} indicate an interesting phenomenon, completely independent of Voisin's proof
of the generic Green Conjecture. They show that Green's Conjecture
for general curves of genus $g=2d-1$ and maximal gonality
$d+1$ is {\em equivalent} to the Gonality Conjecture for
bundles of type $K_C\otimes \OO_C(x+y)$ for general pointed curves
$[C,x,y]\in\mathcal{M}_{2d-2,2}$. We refer to \cite{Ap1} and \cite{AV} for
further implications between the two conjectures, in
both directions.
}
\end{rmk}

\noindent
{\emph{Proof of  Theorem \ref{thm: GL odd genus}.}}
For $C$ as in the hypothesis, and for general points $x,y,z\in C$, we construct a stable curve
$[C']\in\mathcal{M}_{2d+1}$ by adding a smooth rational
component passing through the points $x,y$ and $z$.
Using admissible covers one can show, as in the proofs
of Theorems \ref{thm: small BN} and \ref{thm: GL small BN},
that $C'$ is of maximal gonality, that is $d+2$. From Theorem
\ref{thm: nodal}, we obtain $K_{d,1}(C',\omega_{C'})=0$.
The conclusion follows from the observation:
$K_{d,1}(C,K_C\otimes \OO_C(x+y+z))\cong K_{d,1}(C',\omega_{C'})$.
\hfill $\Box$

It is natural to ask the following:

\begin{question}
For a curve $C$ and points $x, y\in C$,  can one give explicit conditions on Koszul
cohomology ensuring that $x+y$ is contained in a
fiber of a minimal pencil?
\end{question}

We prove here the following result, which can be considered
as a precise version of the Gonality Conjecture for generic curves.

\begin{thm}
 \label{thm: GL}
Let $[C]\in \cM_{2d-2}$, and $x, y\in C$  arbitrarily chosen distinct points. Then
$K_{d-1,1}(C,K_C\otimes \OO_C(x+y))\ne 0$ if and only if there
exists $A\in W^1_{d}(C)$ such that $h^0(C,A\otimes \OO_C(-x-y))\ne 0$.
\end{thm}

\proof
Suppose there exists $A\in W^1_{d}(C)$ such that $h^0(C,A(-x-y))\ne 0$.
 Theorem \ref{thm: nonvan}  applied
to the decomposition $K_C(x+y)=A\otimes B$, with
$B=K_C(x+y)\otimes A^{\vee}$ produces nontrivial Koszul
classes in  the group $K_{d-1,1}(C,K_C\otimes \OO_C(x+y))$.

\medskip

For the converse, we consider $C'$ the stable
curve obtained from $C$ by gluing together the points
$x$ and $y$ and denote by $\nu:C\to C'$ the normalization
morphism. Clearly $[C']\in \mm_{2d-1}$.
We observe that
$$K_{d-1,1}(C,K_C\otimes \OO_C(x+y))\cong K_{d-1,1}(C',\omega_{C'}).$$
From Theorem \ref{thm: nodal}, it follows that $[C']\in \mm_{2d-1, d}^1$,
hence there exists a map
$$f:\widetilde{C}\stackrel{d:1}\longrightarrow R$$ from a curve $\widetilde{C}$ semistably equivalent to $C'$ onto a rational nodal
curve $R$. The curve $C$ is a subcurve of $\widetilde{C}$ and $f_{| C}$ provides the desired pencil.
\endproof

As mentioned above, the lower possible bound for explicit examples
of line bundles that verify the Gonality Conjecture found so far was $2g$.
One can raise the question whether this bound is optimal or not and the sharpest statement one can make is Conjecture
\ref{eta}  discussed in the introduction of this paper.

\section{The Strong Maximal Rank Conjecture}

Based mainly on work carried out in \cite{Fa06a} and \cite{Fa06b} we propose a conjecture predicting the resolution of an embedded curve with general moduli. This statement unifies two apparently
unrelated deep results in the theory of algebraic curves: The \emph{Maximal Rank Conjecture} which predicts the number of hypersurfaces of each degree
containing an general embedded curve $C\subset \mathbb P^r$ 
and \emph{Green's Conjecture} on syzygies of canonical curves.

We begin by recalling the statement of the classical \emph{Maximal Rank Conjecture}. The modern formulation  of this conjecture  is due to Harris \cite{H82} p. 79, even though it appears that traces of a similar statement can be found in the work of Max Noether: We fix integers $g, r$ and $d$
such that $\rho(g, r, d)\geq 0$ and denote by $\mathfrak I_{d, g, r}$ the unique component of the
Hilbert scheme $\mathrm{Hilb}_{d, g, r}$ of curves $C\subset \mathbb P^r$ with Hilbert polynomial $h_C(t)=dt+1-g$, containing curves with general moduli. In other words, the variety $\mathfrak I_{d, g, r}$ is
characterized by the following  properties:

\noindent
(1) The general point $[C\hookrightarrow \mathbb P^r]\in \mathfrak I_{d, g, r}$ corresponds to a smooth curve $C\subset \mathbb P^r$ with $\mbox{deg}(C)=d$ and
$g(C)=g$.
\newline
(2) The moduli map $m:\mathfrak I_{d, g, r}-->\cM_g$, \ $m([C\hookrightarrow \mathbb P^r]):=[C]$ is dominant.

\begin{con}\label{maxrk}\emph{(Maximal Rank Conjecture)}
A general embedded smooth curve $[C\hookrightarrow \PP^r]\in \mathfrak I_{d, g, r}$ is of maximal rank, that is, for all integers $n\geq 1$ the restriction maps
$$\nu_n(C):H^0(\mathbb P^r, \OO_{\mathbb P^r}(n))\rightarrow H^0(C, \OO_C(n))$$
are of maximal rank, that is, either injective or surjective.
\end{con}
Thus if a curve $C\subset \mathbb P^r$ lies on a hypersurface of degree $d$, then either hypersurfaces of degree $d$ cut out the complete linear series $|\OO_C(d)|$ on the curve, or else, $C$ is special in its Hilbert scheme. Since $C$ can be assumed to be a Petri general curve, it follows that $H^1(C, \OO_C(n))=0$ for $n\geq 2$, so $h^0(C, \OO_C(n))=nd+1-g$ and Conjecture \ref{maxrk} amounts to knowing the Hilbert function of $C\subset \mathbb P^r$, that is, the value of $h^0(\mathbb P^r, \I_{C/\mathbb P^r}(n))$ for all $n$.

\begin{ex}
We consider the locus of curves $C\subset \mathbb P^3$ with $\mbox{deg}(C)=6$ and $g(C)=3$ that lie on a quadric surface, that is, $\nu_2(C)$ fails to be an isomorphism.  Such curves must be of type $(2, 4)$ on the quadric, in particular, they are hyperelliptic. This is a divisorial condition on $\mathfrak{I}_{6, 3, 3}$, that is, for a general $[C\hookrightarrow \mathbb P^3]\in \mathfrak I_{6, 3, 3}$ the map $\nu_2(C)$ is an isomorphism.
\end{ex}
Conjecture \ref{maxrk} makes sense of course for any component of $\mathfrak I_{d, g, r}$ but is known to fail outside
the Brill-Noether range, see \cite{H82}. The Maximal Rank Conjecture is known to hold in the non-special range, that is when $d\geq g+r$, due to work of Ballico and Ellia relying on the \emph{m\'ethode d'Horace} of Hirschowitz, see \cite{BE87}. Vosin has also proved cases of the conjecture when $h^1(C, \OO_C(1))=2$, cf. \cite{V92}.  Finally, Conjecture \ref{maxrk} is also known in the case  $\rho(g, r, d)=0$ when it has serious implications for the birational geometry of $\mm_g$. This case can be reduced to the case when  $\mbox{dim } \mbox{Sym}^n H^0(C, \OO_C(1))=\mbox{dim } H^0(C, \OO_C(n))$, that is,
$${n+r\choose n}=nd+1-g,$$
when Conjecture \ref{maxrk} amounts to constructing one smooth curve $[C\hookrightarrow \mathbb P^r]\in \mathfrak I_{d, g, r}$ such that $H^0(\mathbb P^r, \I_{C/\mathbb P^r}(n))=0$.
In this situation,  the failure locus of Conjecture \ref{maxrk} is precisely the virtual divisor $\cZ_{g, 0}$ on $\cM_g$ whose geometry has been discussed in Section 3,  Corollary \ref{maxrankk}. The most interesting case (at least from the point of view of slope calculations) is that of $n=2$. One has the following result \cite{Fa06b} Theorem 1.5:

\begin{thm}\label{mr}
For each $s\geq 1$ we fix integers $$g=s(2s+1),\ \  r=2s \mbox{ and  }\  d=2s(s+1),$$ hence $\rho(g, r, d)=0$.  The locus
$$\mathcal{Z}_{g, 0}:=\{[C]\in \cM_g:\exists L\in W^r_d(C) \mbox{ such that } \nu_2(L): \mathrm{Sym}^2 H^0(C, L)\stackrel{\ncong}\rightarrow H^0(C, L^{\otimes 2}) \mbox{}\}$$
is an effective divisor on $\cM_g$. In particular, a general curve $[C]\in \cM_g$ satisfies the Maximal Rank Conjecture with respect to \emph{all} linear series
$L\in W^r_d(C)$.
\end{thm}

For $s=1$ we have the equality $\cZ_{3, 0}=\cM_{3, 2}^1$, and we recover the hyperelliptic locus on $\cM_3$. The next case, $s=2$ and $g=10$ has been treated in detail in \cite{FaPo05}. One has a scheme-theoretic equality $\cZ_{10, 0}=42\cdot \K_{10}$ on $\cM_{10}$, where $42=\#(W^4_{12}(C))$ is the number of minimal pencils $\mathfrak g^1_6=K_C(-\mathfrak g^4_{12})$ on a general curve $[C]\in \cM_{10}$. Thus a curve $[C]\in \cM_{10}$ fails the Maximal Rank Conjecture for a linear series $L\in W^4_{12}(C)$ if and only if it fails it for all the
$42$ linear series $\mathfrak g^4_{12}$! This incarnation of the $K3$ divisor $\K_{10}$ is instrumental in being able to compute the class of $\kk_{10}$ on $\mm_{10}$, cf. \cite{FaPo05}.
\vskip 3pt
In view of Theorem \ref{mr} it makes sense to propose a much stronger form of Conjecture \ref{maxrk}, replacing the  generality assumption of $[C\hookrightarrow \mathbb P^r]\in \mathfrak I_{d, g, r}$  by a  generality assumption of
$[C]\in \cM_g$ with respect to moduli and asking for the maximal rank of the curve with respect to all linear series
$\mathfrak g^r_d$.
\vskip 7pt
We fix positive integers $g, r, d$ such that $g-d+r\geq 0$ and satisfying $$0\leq \rho(g, r, d)<r-2.$$ We also fix a  
general curve $[C]\in \cM_g$. The numerical assumptions imply that all the linear series $l\in G^r_d(C)$ are complete (the inequality $\rho(g, r+1, d)<0$ is satisfied), as well as very ample. For each (necessarily complete) linear series $l=(L, H^0(C, L))\in G^r_d(C)$ and integer $n\geq 2$, we denote by
$$\nu_n(L): \mathrm{Sym}^n H^0(C, L) \rightarrow H^0(C, L^{\otimes n})$$ the multiplication map of global sections. We then choose a Poincar\'e bundle on $C\times \mbox{Pic}^d(C)$ and 
construct construct two vector bundles $\E_n$ and $\F_n$ over $G^r_d(C)$ with  $\mbox{rank}(\E_n)={r+n\choose n}$ and $\mbox{rank}(\F_n)=h^0(C, L^{\otimes n})=nd+1-g$, together with a bundle morphism $\phi_n:\E_n\rightarrow \F_n$, such that for $L\in G^r_d(C)$ we have that
$$\E_n(L)=\mathrm{Sym}^n H^0(C, L)\ \mbox{ and } \F_n(L)=H^0(C, L^{\otimes n})$$ and $\nu_n(L)$ is the map given by multiplication of global sections.

\begin{con}\label{strongmaxrank} \emph{(Strong Maximal Rank Conjecture)} \
We fix integers $g, r, d\geq 1$ and $n\geq 2$ as above. For a general curve  $[C]\in \cM_g$, the
determinantal variety
$$\Sigma_{n, g, d}^r(C):=\{L\in G^r_d(C): \nu_n(L) \mbox{ is not of maximal rank}\}$$
has expected dimension, that is,
$$\mathrm{dim } \ \Sigma_{n, g, d}^r(C)=\rho(g, r, d)-1-|\mathrm{rank}(\E_n)-\mathrm{rank}(\F_n)|,$$
where by convention, negative dimension means that $\Sigma_{n, g, d}^r(C)$ is empty.
\end{con}
For instance, in the case $\rho(g, r, d)< nd-g-{r+n\choose n}$, the conjecture predicts that
for a general $[C]\in \cM_g$ we have that $\Sigma_{n, g, d}^r(C)=\emptyset$, that is,
$$H^0(\mathbb P^r, \I_{C/\mathbb P^r}(n))=0$$
for \emph{every} embedding $C\stackrel{|L|}\hookrightarrow \mathbb P^r$ given by  $L\in G^r_d(C)$.

When $\rho(g, r, d)=0$ (and in particular whenever $r\leq 3$), using a standard monodromy argument showing the uniqueness the component $\mathfrak I_{d, g, r}$, the Strong Maximal Rank Conjecture is equivalent to Conjecture \ref{maxrk}, and it states that $\nu_n(L)$ is of maximal rank for a general $[C\stackrel{L}\hookrightarrow \mathbb P^r]\in \mathfrak I_{d, g, r}$.

For $\rho(g, r, d)\geq 1$ however, Conjecture \ref{strongmaxrank} seems to be a more difficult question than Conjecture \ref{maxrk} because one requires a way of seeing \emph{all} linear series $L\in G^r_d(C)$ at once.

\begin{rmk} The bound $\rho(g, r, d)<r-2$ in the statement of Conjecture \ref{strongmaxrank} is clearly a necessary 
condition because for linear series $L\in G^r_d(C)$ which are not very ample, the maps $\nu_n(L)$ have no chance of being of maximal rank.
\end{rmk} 

\begin{rmk}
We discuss Conjecture \ref{strongmaxrank} when $r=4$ and $\rho(g, r, d)=1$. The conjecture is trivially true for $g=1$. The first interesting case is $g=6$ and $d=9$. For a general curve $[C]\in \cM_6$ we observe that there is an isomorphism $C\cong W^4_9(C)$ given by $C\ni x\mapsto K_C\otimes \OO_C(-x)$.
Since $\mbox{rank}(\E_2)=15$ and $\mbox{rank}(\F_2)=13$, Conjecture \ref{strongmaxrank} predicts that,
$$\nu_2(K_C(-x)):\mathrm{Sym}^2 H^0\bigl(C, K_C\otimes \OO_C(-x)\bigr)\twoheadrightarrow H^0\bigl(C, K_C^{\otimes 2}\otimes \OO_C(-2x)\bigr),$$ for all $x\in C$, which is true (use the Base-Point-Free Pencil Trick).

The next case is $g=11, d=13$, when the conjecture predicts that the map
$$\nu_2(L):\mathrm{Sym}^2 H^0(C, L)\rightarrow H^0(C, L^{\otimes 2})$$ is injective for all $L\in W^4_{13}(C)$.
This follows (non-trivially) from \cite{M94}. Another case that we checked is $r=5, g=14$ and $d=17$, when $\rho(14, 5, 17)=2$.

\begin{prop}
The Strong Maximal Rank Conjecture holds for general non-special curves, that is, when $r=d-g$.
\end{prop}
\begin{proof}
This is an immediate application of a theorem of Mumford's stating that for any line bundle $L\in \mbox{Pic}^d(C)$ with $d\geq 2g+1$, the map $\nu_2(L)$ is surjective, see e.g. \cite{GL86}. The condition $\rho(g, r,d)<r-2$ forces in the case $r=d-g$ the inequality $d\geq 2g+3$. Since the expected dimension of $\Sigma_{n, g, d}^{d-g}(C)$ is negative, the conjecture predicts that $\Sigma_{n, g, d}^{d-g}(C)=\emptyset$.  This is confirmed by Mumford's result.
\end{proof}

\end{rmk}

\subsection{The Minimal Syzygy Conjecture}
Interpolating between Green's Conjecture for generic curves (viewed as a vanishing statement) and the Maximal Rank Conjecture, it is natural to expect that the Koszul cohomology groups of
line bundles on a general curve $[C]\in \cM_g$ should be subject to the vanishing suggested by the determinantal description provided by Theorem
\ref{specsyzvirt}. For simplicity  we restrict ourselves to the case $\rho(g, r, d)=0$:
\begin{con}\label{specialsyz}
We fix integers $r, s\geq 1$ and set $d:=rs+r$ and $g:=rs+s$, hence $\rho(g, r, d)=0$. For a general curve $[C]\in \cM_g$ and for every integer
$$0\leq p\leq \frac{r-2s}{s+1}$$
we have the vanishing $K_{p, 2}(C, L)=0$, for every linear series $L\in W^r_d(C)$.
\end{con}
As pointed out in Theorem \ref{specsyzvirt}, in the limiting case $p=\frac{r-2s}{s+1}\in \mathbb Z$, Conjecture \ref{specialsyz} would imply that the failure locus
$$\mathcal{Z}_{g, p}:=\{[C]\in \cM_g: \exists L\in W^r_d(C)\ \mbox{ such that }K_{p, 2}(C, L)\neq \emptyset\}$$
is an effective divisor on $\cM_g$ whose closure $\overline{\mathcal{Z}}_{g, p}$ violates the Slope Conjecture.

Conjecture \ref{specialsyz} generalizes Green's Conjecture for generic curves: When $s=1$, it reads like
$K_{p, 2}(C, K_C)=0$ for $g\geq 2p+3$, which is precisely the main result from \cite{Voisin: odd}.  Next, in the case $p=0$, Conjecture \ref{specialsyz} specializes to
Theorem \ref{mr}. The conjecture is also known to hold when $s=2$ and $g\leq 22$ (cf. \cite{Fa06a} Theorems 2.7 and 2.10).


\begin{thebibliography}{ACGH84}
\bibitem[Ap02]{Ap1}
Aprodu, M.:
On the vanishing of higher syzygies of curves.
Math. Zeit., \textbf{241}, 1--15 (2002)

\bibitem[Ap04]{Ap04}
Aprodu, M.:
Green-Lazarsfeld gonality Conjecture for a generic curve of odd genus.
Int. Math. Res. Notices, \textbf{63}, 3409--3414 (2004)

\bibitem[Ap05]{Ap05}
Aprodu, M.:
Remarks on syzygies of $d$-gonal curves.
Math. Res. Lett., \textbf{12}, 387--400 (2005)

\bibitem[ApN08]{Aprodu-Nagel: book}
Aprodu, M., Nagel, J.: Koszul cohomology and algebraic geometry, Max Planck Institut preprint 08-52, (2008).


\bibitem[ApP06]{ApP06}
Aprodu, M., Pacienza, G.:
The Green Conjecture for Exceptional Curves on a K3 Surface.
Int. Math. Res. Notices (2008) 25 pages.

\bibitem[ApV03]{AV}
Aprodu, M., Voisin, C.:
Green-Lazarsfeld's Conjecture for generic curves of large gonality.
C.R.A.S., \textbf{36}, 335--339 (2003)


\bibitem[ACGH85]{ACGH}
Arbarello, E., Cornalba, M., Griffiths, P. A., Harris, J.:
Geometry of algebraic curves, Volume I.
Grundlehren der mathematischen Wissenschaften 267,
Springer-Verlag (1985)

\bibitem[AC81a]{AC81a}
Arbarello, E., Cornalba, M.:
Footnotes to a paper of Beniamino Segre.
Math. Ann., \textbf{256}, 341--362 (1981)

\bibitem[AC81b]{AC81b}
Arbarello E ., Cornalba, M.:
Su una congettura di Petri.
Comment. Math. Helv. \textbf{56}, no. 1, 1--38 (1981)



\bibitem[BE87]{BE87} Ballico, E., Ellia, P.: The maximal rank Conjecture for nonspecial curves in $\mathbb P^n$. Math. Zeit.  \textbf{196},  no.3, 355--367 (1987)

\bibitem[BG85]{BG85}
Boraty\'nsky M., Greco, S.:
Hilbert functions and Betti numbers in a flat family.
Ann. Mat. Pura Appl. (4), \textbf{142}, 277--292 (1985)

\bibitem[Co83]{Co83}
Coppens, M.:
Some sufficient conditions for the gonality of a smooth curve.
J. Pure Appl. Algebra \textbf{30}, no. 1, 5--21 (1983).

\bibitem[CM91]{CM}
Coppens, M., Martens, G.:
Secant spaces and Clifford's Theorem,
Compositio Math., \textbf{78}, 193--212 (1991)


%
\bibitem[Ein87]{Ein}
 Ein, L.:
 A remark on the syzygies of the generic canonical curves.
 J. Differential Geom. \textbf{26}, 361--365 (1987)
%



\bibitem[Ei92]{Ei92} Eisenbud, D.: Green's Conjecture: An orientation for
algebraists. In \emph{Free Resolutions in Commutative Algebra and Algebraic Geometry},
Boston 1992.
\bibitem[Ei06]{Eisenbud: Geom Syzygies}
Eisenbud, D.:
Geometry of Syzygies.
Graduate Texts in Mathematics, 229, Springer Verlag (2006)

\bibitem[EH89]{EH89} Eisenbud, D., Harris, J: {Irreducibility of some  families
of linear series with Brill-Noether number $-1$, } Ann.\
Sci. \'Ecole \ Normale \ Superieure \textbf{22} (1989),
33--53.


\bibitem[EH87]{EH87b}
Eisenbud, D., Harris, J.:
The Kodaira dimension of the moduli space of curves of genus $\geq 23$.
Invent. Math. \textbf{90}, no. 2, 359--387 (1987)

\bibitem[EH86]{EH86}
Eisenbud, D.,  Harris, J., {Limit linear series: basic theory},
Invent. Math. \textbf{85} (1986), 337-371


\bibitem[ELMS89]{ELMS}
Eisenbud, D., Lange, H., Martens, G., Schreyer, F.-O.:
The Clifford dimension of a projective curve.
Compositio Math. \textbf{72}, 173--204 (1989)

\bibitem[EGL]{EGL}
Ellingsrud, G., G\"ottsche, L., Lehn, M.:
On the cobordism class of the Hilbert scheme of a surface.
J. Algebraic Geom. \textbf{10} no. 1, 81-100 (2001).

\bibitem[FaL08]{FaL08}
Farkas, G.,  Ludwig, K.:
The Kodaira dimension of the moduli space of Prym varieties.
Preprint arXiv:0804.4616, to appear in J. European Math. Soc (2009).

\bibitem[FaPo05]{FaPo05}
Farkas, G., Popa, M.:
Effective divisors on $\mm_g$, curves on $K3$ surfaces, and the Slope
Conjecture.
J. Algebraic Geom. \textbf{14}, 241--267 (2005)

\bibitem[Fa00]{Fa00} Farkas, G.: The geometry of the moduli space of curves
of genus $23$, Math. Ann. \textbf{318}, 43-65 (2000).

\bibitem[Fa06a]{Fa06a}
Farkas, G.:
Syzygies of curves and the effective cone of $\mm_g$.
Duke Math. J., \textbf{135}, No. 1, 53--98 (2006)

\bibitem[Fa06b]{Fa06b}
Farkas, G.:
Koszul divisors on moduli space of curves.
Preprint arXiv:0607475 (2006), to appear in American J. Math. (2009)


\bibitem[Fa08]{Fa08}
Farkas, G.:
Aspects of the birational geometry of $\mm_g$.
Preprint arXiv:0810.0702 , to appear in Surveys in Differential Geometry (2009).

\bibitem[Gr84a]{Gr84a}
Green, M.:
Koszul cohomology and the geometry of projective varieties.
J. Differential Geom., \textbf{19}, 125-171 (1984)

\bibitem[Gr84b]{Gr84b}
Green, M.:
Koszul cohomology and the geometry of projective varieties. II.
J. Differential Geom., \textbf{20}, 279--289 (1984)

\bibitem[Gr89]{Gr89}
Green, M.:
Koszul cohomology and geometry.
In: Cornalba, M. (ed.) et al., Proceedings of the first college on Riemann
surfaces held in Trieste, Italy, November 9-December 18, 1987.
Teaneck, NJ: World Scientific Publishing Co. 177--200 (1989)

\bibitem[GL84]{GL84}
Green, M., Lazarsfeld, R.:
The nonvanishing of certain Koszul cohomology groups.
J. Differential Geom., \textbf{19}, 168--170 (1984)

\bibitem[GL86]{GL86}
Green, M., Lazarsfeld, R.:
On the projective normality of complete linear series on an algebraic curve.
Invent. Math., \textbf{83}, 73--90 (1986)


\bibitem[H82]{H82} Harris, J: Curves in projective space,  Presses de l'Universit\'e de Montr\'eal 1982.


\bibitem[HM82]{HM}
Harris, J., Mumford, D.:
On the Kodaira dimension of the moduli space of curves.
Invent. Math., \textbf{67}, 23--86 (1982)

\bibitem[HaM90]{HaMo}
Harris, J., Morrison, I.:
Slopes of effective divisors on the moduli space of stable curves.
Invent. Math., \textbf{99}, 321--355 (1990)

%


\bibitem[HR98]{HR98}
Hirschowitz, A., Ramanan, S.:
New evidence for Green's Conjecture on syzygies of canonical curves.
Ann. Sci. \'Ecole Norm. Sup. (4), \textbf{31}, 145--152 (1998)


\bibitem[Ke90]{Ke90} Keem, C.:
 On the variety of special linear systems on an algebraic curve.
 Math. Ann. \textbf{288}, 309--322 (1990)





%

\bibitem[La89]{La1}
Lazarsfeld, R.:
A sampling of vector bundle techniques in the study of linear series.
In: Cornalba, M. (ed.) et al., Proceedings of
the first college on Riemann surfaces held in Trieste, Italy, November
9-December 18, 1987. Teaneck, NJ: World Scientific Publishing Co.
500--559 (1989)

\bibitem[Lo89]{Lo}
Loose, F.:
On the graded Betti numbers of plane algebraic curves.
Manuscripta Math., \textbf{64}, 503--514 (1989)

\bibitem[Ma82]{Ma1}
Martens, G.:
\"Uber den Clifford-Index algebraischer Kurven.
J. Reine Angew. Math., \textbf{336}, 83--90 (1982)

%
%

%
 \bibitem[M94]{M94}
 Mukai, S.:
 Curves and $K3$ surfaces of genus eleven. Moduli of vector bundles (Sanda 1994; Kyoto 1994), 189-197, Lecture Notes in Pure and Appl. Math. 1996.

%
%
%
%







 \bibitem[PR88]{Paranjape-Ramanan}
 Paranjape, K., Ramanan, S.:
 On the canonical ring of a curve.
 Algebraic geometry and commutative algebra, Vol. II, 503--516, Kinokuniya,
 Tokyo, 1988.


%




\bibitem[Sch86]{Sc1}
Schreyer, F.-O.:
Syzygies of canonical curves and special linear series.
Math. Ann., \textbf{275}, 105--137 (1986)

\bibitem[Sch89]{Sc2}
Schreyer, F.-O.:
Green's Conjecture for general $p$-gonal curves of large genus.
Algebraic curves and projective geometry, Trento, 1988,
Lecture Notes in Math., \textbf{1389}, Springer, Berlin-New
York 254--260 (1989)

\bibitem[Sch91]{Sc3}
Schreyer, F.-O.:
A standard basis approach to syzygies of canonical curves.
J. Reine Angew. Math., \textbf{421}, 83--123 (1991)

\bibitem[St98]{St98} Steffen, F.: A generalized principal ideal Theorem with applications to Brill-Noether theory. Invent. Math. \textbf{132} 73-89 (1998)

\bibitem[Tei02]{Teixidor02}
Teixidor i Bigas, M.:
Green's Conjecture for the generic $r$-gonal curve of genus $g\geq 3r-7$.
Duke Math. J., \textbf{111}, 195--222 (2002)

%
\bibitem[V88]{Voisin: tetragonales}
Voisin, C.:
Courbes t\'etragonales et cohomologie de Koszul.
J. Reine Angew. Math., \textbf{387}, 111--121 (1988)

\bibitem[V93]{Voisin: Proc LMS}
 Voisin, C.:
 D\'eformation des syzygies et th\'eorie de Brill-Noether.
 Proc. London Math. Soc. (3)  \textbf{67} no. 3, 493--515 (1993).

\bibitem[V02]{Voisin: even}
Voisin, C.:
Green's generic syzygy Conjecture for curves of even genus lying on a $K3$
surface.
J. European Math. Soc., \textbf{4}, 363--404 (2002).

\bibitem[V92]{V92}
Voisin, C.:
Sur l'application de Wahl des courbes satisfaisant la condition de Brill-Noether-Petri.
Acta Mathematica \textbf{168}, 249-272 (1992).

\bibitem[V05]{Voisin: odd}
Voisin, C.:
Green's canonical syzygy Conjecture for generic curves of odd genus.
Compositio Math., \textbf{141} (5), 1163--1190 (2005).



\end{thebibliography}
\end{document}